\documentclass{amsart}
\usepackage{amsopn}
\usepackage{amssymb, amscd, amsmath}
\usepackage{array}
\usepackage{xcolor} 
\usepackage{hyperref}
\usepackage[shortlabels]{enumitem}
\usepackage{tikz-cd}
\usepackage{cleveref}

\newcommand{\nc}{\newcommand}

\DeclareRobustCommand{\SkipTocEntry}[4]{}

\nc{\uv}{{\underline v}}
\nc{\uh}{{\underline h}}
\nc{\mcvu}{{\underline \mcv}}

\nc{\vbph}{{\hat v}_{\beta^+}}
\nc{\vbpu}{{\underline v}_{\beta^+}}
\nc{\vbh}{{\hat v}_\beta}
\nc{\hvbp}{{\hat v}_{\beta^+}}
\nc{\hvb}{{\hat v}_\beta}
\nc{\vbp}{v_{\beta^+}}
\nc{\vb}{v_\beta}

\nc{\g}{h}     
\nc{\vbM}{{\log v_\beta}}  
\nc{\vbg}{\log v_\betab}   
\nc{\bM}{{^\ngo\beta}}      
\nc{\bkg}{\bar \g}          
\nc{\lb}{[\cdot \,,\cdot]}

\nc{\Gl}{\mathsf{GL}}   \nc{\Or}{\mathsf{O}}    \nc{\SO}{\mathsf{SO}}   \nc{\Sl}{\mathsf{SL}}  \nc{\SU}{\mathsf{SU}}
\nc{\G}{\mathsf{G}}     \nc{\K}{\mathsf{K}}     \nc{\Ll}{\mathsf{L}}    \nc{\Bb}{\mathsf{B}}  \nc{\A}{\mathsf{A}}  
\nc{\Hg}{\mathsf{H}}    \nc{\Ii}{\mathsf{I}}    \nc{\N}{\mathsf{N}}     \nc{\Q}{\mathsf{Q}} \nc{\T}{\mathsf{T}}  
\nc{\Ss}{\mathsf{S}} \nc{\U}{\mathsf{U}} \nc{\B}{\mathsf{B}}  \nc{\Zz}{\mathsf{Z}} \nc{\M}{\mathsf{M}} \nc{\F}{\mathsf{F}}
\nc{\Qb}{\mathsf{Q}_\Beta} \nc{\Hb}{\mathsf{H}_\Beta} \nc{\Ub}{\mathsf{U}_\Beta} 
\nc{\Gb}{\mathsf{G}_\Beta} \nc{\Kb}{\mathsf{K}_\Beta} \nc{\Hh}{\mathsf{H}}

\nc{\Xg}{\mathfrak{X}} \nc{\ggo}{\mathfrak{g}}
\nc{\fg}{\mathfrak{f}}  \nc{\zg}{\mathfrak{z}} \nc{\ngo}{\mathfrak{n}} \nc{\kg}{\mathfrak{k}} 
\nc{\mg}{\mathfrak{m}} \nc{\bg}{\mathfrak{b}}  \nc{\sog}{\mathfrak{so}} \nc{\sug}{\mathfrak{su}} 
\nc{\slg}{\mathfrak{sl}} \nc{\glg}{\mathfrak{gl}} \nc{\rg}{\mathfrak{r}}  \nc{\hg}{\mathfrak{h}} 
\nc{\tgo}{\mathfrak{t}} \nc{\ug}{\mathfrak{u}} \nc{\dg}{\mathfrak{d}} \nc{\ag}{\mathfrak{a}} 
\nc{\pg}{\mathfrak{p}} \nc{\sg}{\mathfrak{s}} \nc{\affg}{\mathfrak{aff}} \nc{\qg}{\mathfrak{q}} 
\nc{\lgo}{\mathfrak{l}} \nc{\tg}{\mathfrak{t}}  \nc{\eg}{\mathfrak{e}}  \nc{\vg}{\mathfrak{v}}

\nc{\pca}{\mathcal{P}} \nc{\nca}{\mathcal{N}} \nc{\lca}{\mathcal{L}} \nc{\oca}{\mathcal{O}} \nc{\mca}{\mathcal{M}} \nc{\tca}{\mathcal{T}} \nc{\aca}{\mathcal{A}} \nc{\cca}{\mathcal{C}} \nc{\gca}{\mathcal{G}} \nc{\sca}{\mathcal{S}} \nc{\hca}{\mathcal{H}} \nc{\bca}{\mathcal{B}} \nc{\dca}{\mathcal{D}} \nc{\fca}{\mathcal{F}} \nc{\Qca}{\mathcal{Q}} \nc{\Eca}{\mathcal{E}} \nc{\Kca}{\mathcal{K}} \nc{\rca}{\mathcal{R}}   \nc{\vca}{\mathcal{V}} \nc{\wca}{\mathcal{W}}

\nc{\dd}{{\rm d}}  \nc{\ddt}{\tfrac{{\rm d}}{{\rm d}t}}        \nc{\dds}{\tfrac{{\rm d}}{{\rm d}s}} 
\nc{\ddtbig}{\frac{{\rm d}}{{\rm d}t}}      \nc{\dpar}{\tfrac{\partial}{\partial t}}    
\nc{\im}{\mathtt{i}}    

\nc{\RR}{{\mathbb R}} \nc{\HH}{{\mathbb H}} \nc{\CC}{{\mathbb C}} \nc{\ZZ}{{\mathbb Z}}
\nc{\FF}{{\mathbb F}} \nc{\NN}{{\mathbb N}} \nc{\QQ}{{\mathbb Q}} \nc{\KK}{{\mathbb K}}

\nc{\vs}{\vspace{.2cm}} 
\nc{\ip}{{\langle \,\cdot \,,\cdot \,\rangle }} \nc{\la}{\langle} \nc{\ra}{\rangle} 
\nc{\unm}{\tfrac{1}{2}} \nc{\unc}{\tfrac{1}{4}} 

\nc{\minimatrix}[4]{\left(\begin{smallmatrix} {#1} & {#2} \\ {#3} & {#4} \end{smallmatrix}\right)}
\nc{\twomatrix}[4]{\left[\begin{array}{cc} {#1} & {#2} \\ {#3} & {#4} \end{array} \right]}
\nc{\threematrix}[9]{\left[\begin{array}{ccc} {#1} & {#2} & {#3} \\ {#4} & {#5} & {#6}\\ {#7} & {#8} & {#9} \end{array} \right]}

\nc{\alert}{\color{blue}}

\nc{\ad}{\operatorname{ad}}  \nc{\Aut}{\operatorname{Aut}}   \nc{\Inn}{\operatorname{Inn}}   \nc{\Lie}{\operatorname{Lie}} \nc{\Ad}{\operatorname{Ad}} \nc{\Der}{\operatorname{Der}} \nc{\rad}{\operatorname{rad}} \nc{\kf}{\operatorname{B}}
\nc{\End}{\operatorname{End}} \nc{\rank}{\operatorname{rank}} \nc{\Ker}{\operatorname{Ker}} \nc{\tr}{\operatorname{tr}} \nc{\Sym}{\operatorname{Sym}} \nc{\diag}{\operatorname{diag}} \nc{\proj}{\operatorname{pr}} \nc{\Id}{{\operatorname{Id}}} \nc{\Span}{\operatorname{span}}

\nc{\grad}{\operatorname{grad}} \nc{\divg}{\operatorname{div}} \nc{\divgb}{\overline{\operatorname{div}}}
\nc{\Diff}{\operatorname{Diff}} \nc{\Isom}{\operatorname{Isom}}
\nc{\Ric}{\operatorname{Ric}}   \nc{\ric}{\operatorname{ric}} 
\nc{\Riem}{\operatorname{Rm}}   \nc{\scal}{\operatorname{scal}} 
\nc{\scalm}{\operatorname{scal}^\star} \nc{\Riccim}{\operatorname{Ric}^{\star}}  
\nc{\vol}{\operatorname{vol}}   \nc{\inj}{\operatorname{inj}}
\nc{\isog}{\mathfrak{iso}}      \nc{\ho}{\mathcal{H}} \nc{\ve}{{\mathcal{V}}}      
\nc{\mm}{\operatorname{M}} \nc{\mmm}{\operatorname{m}} 
\nc{\mcv}{{\operatorname{H}}}   \nc{\Hess}{\operatorname{Hess}}

\theoremstyle{plain}
\newtheorem{theorem}{Theorem}[section]
\newtheorem{conjecture}[theorem]{Conjecture}
\newtheorem{proposition}[theorem]{Proposition}
\newtheorem{corollary}[theorem]{Corollary}
\newtheorem{lemma}[theorem]{Lemma}
\newtheorem{teointro}{Theorem}

\theoremstyle{definition}
\newtheorem{definition}[theorem]{Definition}

\newtheorem{assumption}[theorem]{Assumption}

\theoremstyle{remark}
\newtheorem{remark}[theorem]{Remark}

\newtheorem{example}[theorem]{Example}

\AtBeginDocument{%
   \def\MR#1{}
}

\begin{document}
\begin{titlepage}

\title{Non-compact Einstein manifolds with unimodular isometry group}

\author{Christoph B\"ohm}	
\address{University of M\"unster, Einsteinstra{\ss}e 62, 48149 M\"unster, Germany}
\email{cboehm@math.uni-muenster.de}
\author{Ramiro A.~ Lafuente} 
\address{School of Mathematics and Physics, The University of Queensland, St Lucia QLD 4072, Australia}
\email{r.lafuente@uq.edu.au}

\begin{abstract}
We show that a negative Einstein manifold admitting a proper isometric action of a connected unimodular Lie group with compact, possibly singular, orbit space 
splits isometrically as a product of a symmetric space and a compact negative Einstein manifold.  The proof involves 
the theory of polar actions,  Lie-theoretic arguments and maximum principles.
\end{abstract}


\end{titlepage}

\maketitle

\setcounter{tocdepth}{1}

\setcounter{page}{1}

\tableofcontents

\section{Introduction}

Let $(M^n,g)$ be a  complete,  non-compact Einstein manifold  with cocompact symmetries,  meaning that a connected Lie group $\Ll$ acts properly and isometrically on $(M^n,g)$ with a compact orbit space.  Note that $M$, and hence also $\Ll$, must be non-compact by Bochner's theorem \cite{Bochner1948}. 
Our main result is the following  `splitting theorem' when the isometry group is unimodular:

\begin{teointro}\label{thm_unimod}
Let $(M^n,g)$ be a complete Einstein manifold with negative scalar curvature, and let $\Ll$ be a connected unimodular Lie group acting cocompactly, properly and isometrically.  Then, $\Ll$ is semisimple, and there is an isometric splitting of $M$ as $B \times \Ll/\K$, where $B$ is a compact Einstein manifold   and $\Ll/\K$ is a symmetric space of non-compact type.
\end{teointro}

A connected Lie group $\Ll$ is unimodular if and only if 
its left Haar measure is also right invariant.
 Nilpotent, semisimple and  reductive Lie groups
are unimodular. Concerning examples of compact Einstein manifold with 
negative scalar curvature we refer to 
\cite{Aub78,Yau, An06,Ba12,FP20}. Moreover,
let us  mention  that
in the particular case of homogeneous manifolds (i.e.~ when $\Ll$ acts transitively), \Cref{thm_unimod} follows 
from the recent proof of the Alekseevskii conjecture in \cite{alek_sol}.

The first conclusion in Theorem \ref{thm_unimod},
 that $\Ll$ must be semisimple, was known
for  homogeneous spaces  \cite{Dtt88},
and more generally under  the assumption that all orbits are principal  \cite{alek_sol}, i.e.  when 
the orbit space $\Ll \backslash M$ is smooth. We would like to emphasize that,
  in general,
the orbit space can be a non-smooth  metric space, due to the presence of singular and exceptional orbits in $M$. 
The second conclusion ---the isometric splitting--- is new even if the orbit space is smooth.

Assuming  compactness of the orbit space in Theorem \ref{thm_unimod}
is  necessary:
the horospheres in real hyperbolic space $\HH^n$
are the orbits of a proper, isometric, $\RR^{n-1}$-action, 
with orbit space $\RR$, but clearly they do 
 not split isometrically. More generally, any  negative Einstein homogeneous space can be seen as the orbit of an isometric action in an irreducible symmetric space \cite{Jab18}.  Further inhomogeneous examples may be found in    \cite{Bo1999,CP02, DW11,BDW2015,CDJL2020,Wink21}  
 and \cite{CS2021,MR22,AT23},
with compact and  non-compact symmetry group, respectively. 

Regarding the unimodularity assumption, we propose the following

\begin{conjecture}\label{conj_split}
Let $(M^n,g)$ be a complete 
negative Einstein manifold admitting a proper, isometric, cocompact action by a connected Lie group. Then, $M^n$ splits isometrically  as a product of a compact Einstein manifold  and a non-compact homogeneous Einstein space.
\end{conjecture}

In the Ricci flat case, from the Cheeger-Gromoll splitting theorem one can deduce the following:

\begin{teointro}\label{thm_rf_intro}
Let $(M^n,g)$ be a complete, Ricci flat Riemannian manifold, and let $\G$ be a Lie group acting cocompactly, properly and  isometrically. 
Then, the universal cover  splits isometrically as a Riemannian product $\tilde M = \bar M \times \RR^k$, with $\bar M$ compact.  
\end{teointro}

For an effective action it also follows that  $\G_0 \subset \Isom(\RR^k)$.
Concerning examples of compact Ricci flat manifolds we refer to \cite{Joy07}.
Moreover, let us recall that homogeneous Ricci flat manifolds are flat \cite{AlkKml}.

Assuming compactness of the orbit space is again necessary: already in dimension
four there exist toric gravitational instantons (i.e.~ Ricci flat and with $T^2$-symmetry), 
the Riemannian Kerr family (including Schwarzschild's  metric) and the Chen-Teo family:  see \cite{CT11,BG23} and references therein. 
Concerning cohomogeneity one examples, see \cite{D-W,Bo1999,BDW2015, CH22}.

Regarding the sign of the Einstein constant, in the positive case numerous examples of non-product Einstein manifolds with symmetry are known, e.g.~\cite{Page78,Bohm98,W-W, FH17,Wang2012}.

We now briefly outline the proof of \Cref{thm_unimod}. By the Slice Theorem, there is an open 
and dense subset $M_{\sf reg} \subset M$ on which the quotient map $\pi^\Ll : M \to  \Ll\backslash M$ is a Riemannian submersion. Using O'Neill's curvature formulas, the Einstein equation becomes  a coupled system of
 elliptic PDE's, involving the Lie-theoretic data of the group $\Ll$. 
 We analise this system by applying the maximum principle to various functions which satisfy interesting differential inequalities. 
 Note that due to the possible non-smoothness of the 
orbit space  a maximum principle can only be applied to functions
which are known to admit a local maximum in the interior (smooth) part
$\pi^\Ll(M_{\sf reg})$ of  $\Ll\backslash M$ (and not at  boundary points). 
We would like to stress  that  this is not just a technicality
but a serious issue, which we were only able to overcome in case that $\Ll$ is
unimodular: see Lemma \ref{lem_hvbp}.

As a first step, we prove  in \Cref{thm_ricneg}
that $\Ll$ must be semisimple,  assuming only that the Ricci curvature is negative. 
To overcome the non-smoothness of the orbit space, we argue on the frame bundle 
$\pi : FM \to M$, which is a principal $\Or(n)$-bundle. The  natural lifted action of $\Ll$ on $FM$ is proper, isometric and free, thus the quotient $\Ll \backslash FM$ is a smooth Riemannian manifold with a right isometric and proper $\Or(n)$-action for which $\Ll \backslash FM /\Or(n) \simeq \Ll\backslash M$. Arguing as in the proof of \cite[Thm.~7.1]{alek_sol}, we consider the so-called $\beta^+$-volume density, a certain weighted determinant of the orbit metrics; 
if $\Ll$ is abelian, this is precisely the relative volume density of the orbits, cf.~\cite{Rong98,Lot20}. The crucial difference here is that the volume density must be computed with respect to the pull-back metric $\pi^* g$, which is only positive semi-definite. At present we are only able to extend 
the $\beta^+$-volume density functional to possibly degenerate metrics,
assuming  unimodularity of $\Ll$.

In a second step, we consider all possible Iwasawa decompositions $\Ll = \K \A\N$ for the semisimple Lie group $\Ll$. By applying the results of  \cite{alek_sol} to the action of each of the \emph{Borel subgroups} $\G := \A \N$, we are able to deduce that the action of $\Ll$ is polar (i.e.~the distribution $(\ker \pi^\Ll)^\perp$ is integrable). Next, we apply a generalisation of the algebraic curvature estimates from \cite[$\S$12]{alek_sol} that holds without homogeneity, and show that the (identity components of the) minimal parabolic subgroups $\Q \leq \Ll$ act transitively on $\Ll$-orbits. If $\Ll$ is split-semisimple (e.g.~$\Ll= \Sl(n,\RR)$) then $\Q = \G$, and it follows that the $\Ll$-orbits are symmetric spaces. The same is true in general, but to prove this we perform a delicate analysis of the Lie-algebraic structure of the orbits, and  use a Bochner-type argument inspired by the work of Jablonski and Petersen \cite{JblPet14} in the homogeneous case. In the final step, following ideas in \cite{NT18}, we apply the Bochner technique for Killing fields restricted to a section of the $\Ll$-action, and  we conclude that the orbits are totally geodesic.


The short proof 
of Theorem \ref{thm_rf_intro}
 is just a standard application of the Cheeger-Gromoll splitting theorem \cite{ChGr71} and the Bochner technique for Killing fields:
see $\S$\ref{sec_Ricciflat}.

\vs


The article is organised as follows.  In $\S$\ref{sec_Einstein_sym} we review some known results on the Einstein equation in the presence of symmetries (including the main technical results from  \cite{alek_sol}). In $\S$\ref{sec_vol}, we introduce weighted determinants and their corresponding volume densities on the orbits of an isometric action. The proof of \Cref{thm_unimod} is done step by step in $\S\S$4--8. \Cref{thm_rf_intro} is proved in $\S$\ref{sec_Ricciflat}. Finally, we also include some appendices in which we review some basic terminology about proper isometric actions and polar actions (\Cref{app_isom_actions}), the frame bundle geometry (\Cref{sec_FM}) and some important properties of semisimple Lie groups (\Cref{app_sslie}) .

For most of the article, the notation will be consistent with  \cite{alek_sol}, 
except for the mean curvature vector of the orbits, which we will denote by $\mcv$ (and not $N$). Recall that we denote  shape operators by $L_X$, and Lie derivatives by $\lca_X$.

\section{The Einstein equation with symmetries}\label{sec_Einstein_sym}

Let $\G$ be a connected Lie group acting properly, effectively and isometrically on a complete Riemannian manifold $(M^n,g)$. We refer the reader to \Cref{app_isom_actions} for basic properties about these type of actions. The Slice Theorem implies that the orbit space $B = \G \backslash M_{\sf reg}$ corresponding to the regular part $M_{\sf reg}$ is a smooth manifold.  O'Neill's curvature formulas  yield  an expression for the Ricci curvature in vertical direcions, 
\begin{equation}\label{eqn_EV}
   \ric_g (U,U) =  \ric^\G(U,U) + \la  L_{\mcv^\G} U, U \ra + \Vert AU\Vert^2 - \sum_{i} \left\la \left( \nabla_{X_i} L \right)_{X_i} U, U \right\ra,
\end{equation}
in horizontal directions,
\begin{equation}\label{eqn_EH}
 \ric_g(X,X) =  \ric^B(X,X) - 2 \Vert A_X \Vert^2 - \Vert L_X\Vert^2 + \la \nabla_X \mcv^\G , X \ra,
\end{equation}
and also an off-diagonal equation which we will not use in this paper. See \cite[Thm.~2.2]{alek_sol}. 
Regarding notation, $\{ X_i \}$ is any horizontal local orthonormal frame; $U \in \Gamma(\ve^\G)$ (resp.~ $X \in \Gamma(\ho^\G)$) denotes a vertical (resp.~horizontal) vector field; $A = A^\G$ is O'Neill's integrability tensor; $L_X$ the shape operator of the orbits in the direction of $X$ and $\mcv^\G$ their mean curvature vector;  $\ric^\G$ denotes the Ricci curvature of the orbits, $\ric^B$ that of the quotient metric on $B$; and finally, the pointwise norms are computed via
\[
    \Vert AU \Vert :=  \sum_i \Vert A_{X_i} U\Vert^2, \qquad \Vert A_X \Vert^2 := \sum_i \Vert A_X X_i\Vert^2 = \sum_r \Vert A_X U_r \Vert^2, \qquad \Vert L_X \Vert^2 := \sum_r \Vert L_X U_r \Vert^2,
\] 
where $\{ U_r\}$ is any vertical orthonormal frame.

We denote by $\N\lhd \G$  the nilradical of $\G$, and we use the superscript $\N$ to denote objects corresponding to the Riemannian submersion associated with the induced action of $\N$ on $M$:
\[
  \pi^\N : M \to \N\backslash M =: P.
\]
Assuming that $\N$ acts freely and polarly, $\pi^\N$ becomes a principal $\N$-bundle, and  there exists a section $\sigma : \Sigma_\N \to M$ for the $\N$-action with polar group $\Pi := \N_\Sigma$. In this context, the $\N$-invariant metric $g$ on $M$ can be characterised by  the following data: 
\begin{itemize}
  \item a Riemannian metric $g^P$ on $P$;
  \item a flat principal connection $\ho^\N$ on $M \to \N\backslash M$;
  \item a $\Pi$-equivariant map  $h : \Sigma_\N \to \Sym_+^2(\ngo^*)$, defined by 
  \[
      s\mapsto h(s) : \ngo \times \ngo \to \RR, \qquad  h(s)(U,V) := g(U^*, V^*)_{\sigma(s)}.
  \]
  where $U^*, V^*$ denote the Kiling fields on $M$ associated to $U,V \in \ngo$.
\end{itemize}  
In the last bullet point, the action of $\Pi \leq \N$ on inner products is given by 
\[
  \rho \cdot \ip = \left\la \Ad(\rho^{-1}) \, \cdot \, , \Ad(\rho^{-1}) \, \cdot \, \right\ra, \qquad \rho \in \Pi, \quad \ip \in \Sym_+^2(\ngo^*).
\]
By abuse of notation,  we also denote by $g_p^\N$ the metric induced by $g$ on the orbit $\N \cdot p$.
For the most general version of the above characterisation of invariant metrics (without assuming that the action is free and polar)  we refer the reader to \cite{BH87}.

\begin{remark}
The \emph{quotient metric} $g^P$  is characterised by the fact that $\pi^\N$ is a Riemannian submersion. Also, recall that after fixing a background inner product $\ip$ on $\ngo$, we have an identification $\Sym^2_+(\ngo^*) \simeq \Gl^+(\ngo) / \SO(\ngo,\ip)$,  a symmetric space of non-compact type. 
\end{remark}

The main technical results of \cite{alek_sol} essentially state the negative Einstein equation with cocompact $\G$-symmetry implies that $\N$ acts polarly, with Ricci soliton orbits, $h$ is a harmonic map between Riemannian manifolds \cite{ES64} (considering on $\Sym^2_+(\ngo^*)$ the symmetric metric), and  the quotient metric satisfies an Einstein-like equation. More precisely:

\begin{theorem}\cite{alek_sol}\label{thm_Naction}
Let $(M^n,g)$, $\G$ be as above, assume that $(M^n,g)$ is Einstein with $\ric_g = -g$, $M_{\sf reg} = M$, and that $B = \G\backslash M$ is compact. Then, the induced isometric action of the nilradical $\N \leq \G$ on $(M^n, g)$ satisfies the following:
\begin{enumerate}[(i)]
    \item The   action of $\N$ is polar (equivalently, $\ho^\N$ is flat: $A^\N = 0$); \label{item_Npolar}
    \item  The $\N$-orbits are locally isometric to nilsolitons, and we have \label{item_Nnilsol}
    \begin{equation}\label{eqn_RicveN}
        \ric^{\N} = - g^{\N} - \tfrac12 \,  \lca_{\mcv^\N} \left(g^{\N} \right), \qquad \scal^\N = -\dim \ngo +\tr \beta^+, \qquad L_{\mcv^\N} = - \beta^+.
    \end{equation}
    \item For any $\N$-section $\Sigma_\N$, the `orbit metrics map' $h : \Sigma_\N \to \Sym^2_+(\ngo^*) \simeq \Gl^+(\ngo)/\SO(\ngo)$ is a harmonic map; \label{item_HM}
    \item The Ricci curvature of the base satisfies \label{item_RicP}
    \[
        \ric^P = -g^P + L^* L,
    \]
    where $L^* L \in \End(TP)$ is given by $\la (L^* L) X, Y \ra = \tr (L_X \, L_Y)$, $X,Y\in \Gamma(TP)$.
\end{enumerate}
If in addition $\G$ is a completely solvable Lie group admitting a non-flat left-invariant Einstein metric, and acts almost-freely on $M$, then:
\begin{enumerate}[(i)]
    \setcounter{enumi}{4}
    \item  $\mcv^\N$ is tangent to the $\G$-orbits.  \label{item_HN_Gvert}
\end{enumerate}
\end{theorem}

\begin{proof}
\Cref{item_Npolar} is \cite[Thm.~8.1(i)]{alek_sol}, and \Cref{item_Nnilsol} follows from \cite[Thm.~8.1(iii), Prop.~9.7 $\&$ Cor.~9.8]{alek_sol}. For \Cref{item_HM}, using the information from the previous two items in \eqref{eqn_EV} (for the $\N$-submersion) yields 
\[
    \sum_i \left(\nabla_{X_i} L\right)_{X_i} = 0.
\]
for a local, orthonormal, $\N$-horizontal frame $\{X_i\}$. Rewriting this tensorial expression in local coordinates, the claim follows from \cite[Prop.~4.17]{Lot07}.  \Cref{item_RicP} follows from \eqref{eqn_EH} using that $A^\N = 0$ and $\nabla \mcv^\N |_{\ho^\N} = 0$ (\cite[Thm.~8.1., (ii)]{alek_sol}, cf.~also \cite[Cor.~9.9]{alek_sol}). 
 Finally, \Cref{item_HN_Gvert} is \cite[Thm.~11.3]{alek_sol}.
\end{proof}

\begin{remark}
When $\N$ is simply connected, there exists a vertical vector field $X_{\beta^+} \in \Gamma(\ve^\N)$ which on each orbit satisfies $\lca_{\mcv^\N}(g^\N) = \lca_{X_{\beta^+}}(g^\N)$. Thus, in this case, the orbits are indeed Ricci solitons (cf.~\cite{soliton}). In general they are only locally isometric to one, since the soliton vector field does not commute with deck transformations \cite{solvsolitons}. 
\end{remark}


\section{Volume densities on Riemannian manifolds with symmetries}\label{sec_vol}

The goal of this section is to introduce a set of weighted volume densities for the orbits of an isometric group action in a complete Riemannian manifold. These will play a crucial role in the proof of our main results. They were first considered in \cite{BL17} (in the homogeneous case), and later refined in \cite{alek_sol}. Essentially, they generalise the usual Riemannian volume density function $\sqrt{\det g_{ij}}$ by introducing weights. However, the need for `fixing the right gauge' makes them very hard to define explicitly in terms of the metric coefficients $g_{ij}$. This is illustrated by the simplest non-trivial case in Example \ref{ex_heis}. Further complications arise when studying their behavior near singular orbits, a case which was not treated in \cite{alek_sol}. To overcome these issues, we use a parametrisation of the space of inner products \eqref{eqn_qcdoth} by  lower-triangular matrices, that allows us to introduce the `weighted determinant' function.

\subsection{The volume density of the orbits}\label{sec_volume_density}

Consider a proper isometric action of a unimodular Lie group $\N$ on a Riemannian manifold $(M^n,g)$, with discrete principal isotropy groups. A choice of basis $\{U_i\}_{i=1}^{\dim \ngo}$ for $\ngo$ yields, on the principal part, a global vertical frame $\{U_i^*\}_{i=1}^{\dim \ngo}$ consisting of Killing fields. We define the volume density of the orbits via $v_\N : M \to \RR_{>0}$,
\[
    v_\N (p) := \sqrt{\det g_{ij}(p)},  \qquad g_{ij}(p) := g(U_i^*, U_j^*)_p.
\]
Notice that $v^2_\N$ is well-defined and smooth on all of $M$, whereas $v_\N$ is only smooth on the regular part $M_{\sf reg}$, i.e.~where the Killing fields do not vanish. Of course $v_\N$ is still a continuous function on all of $M$, and it vanishes precisely on singular orbits.

\begin{lemma}\label{lem_vN_Ninv}
If $\N$ is unimodular then the volume density $v_\N$ is $\N$-invariant. 
\end{lemma}
\begin{proof}
For $q = x \cdot p$, $x\in \N$, we have
\[
         (dL_x)_p (U^*)_p = (\Ad(x) U)^*_{q},
\]
see e.g.~\cite[Appendix A]{BL18}, thus using that $L_x$ is an isometry we obtain
\[
    g_{ij}(p) =  g( (dL_x)_p (U_i^*)_p, (dL_x)_p (U_j^*)_p) = g((\Ad(x) U_i)^*, (\Ad(x) U_j)^*)_q.
\]
Unimodularity now implies that
\[
    v_\N (p) = (\det \Ad(x)) v_\N (q) = v_\N (q).
\]
\end{proof}

Notice that \Cref{lem_vN_Ninv} does not hold if $\N$ is not unimodular. 

We say an isometry $f\in \Isom(M,g)$ normalises $\N$, if the image of $\N$ in $\Isom(M,g)$ under the action map is normalised by $f$. In this case, $f_*$ maps $\N$-Killing fields to $\N$-Killing fields, and hence it defines a linear map
\[
    \Ad(f)|_\ngo : \ngo \to \ngo,
\]
where $\Ad$ is the adjoint representation of $\Isom(M,g)$. The same argument as in \Cref{lem_vN_Ninv} shows:

\begin{lemma}\label{lem_vN-finv}
If $f\in \Isom(M,g)$ normalises $\N$ then  
\[
        v_\N(f(p)) = \det \left(\Ad(f)|_\ngo\right)^{-1}  v_\N(p).
\]
\end{lemma}

It is well-known (see e.g.~ \cite[Lemma 3.2]{alek_sol}) that the mean curvature vector $\mcv^\N$ of the principal $\N$-orbits is given by
\begin{equation}\label{eqn_mcv_logvn}
    \mcv^\N = -\nabla \log v_\N.
\end{equation}
It also follows from a standard computation that  the first-variation of $ \log v_\N$ is 
\[
    X(\log v_\N) = \tr L_X, \qquad X\in \ho^\N. 
\]

\subsection{The weigthed determinant}

Our goal in this section is to introduce  an extension of the determinant function on $n\times n$ real matrices.  Recall that the absolute value of the usual determinant function, $|{\det}| : \Gl_n(\RR) \to \RR^+$, could be characterised as the unique Lie group homomorphism whose differential at the identity is the trace $A \in \glg_n(\RR) \mapsto \tr A$ (the latter is a Lie algebra homomorphism since commutators are traceless).

Let now $W \in \glg_n(\RR)$ be a fixed diagonal matrix, the `weights'. We have the corresponding subgroups of $\Gl_n(\RR)$: $\G_W, \U_W, \Q_W , \K_W$ defined in \cite[$\S$4]{BL17} (cf.~also \cite[$\S$7]{GIT20}, \cite[$\S$4]{HSS}), and their  Lie algebras denoted by the corresponding gothic letters. Roughly speaking, if $V_1 \oplus \cdots \oplus V_l$ is the eigenspace decomposition of $W$, corresponding respectively to the different eigenvalues $w_1 < \cdots < w_l$, then $\ug_W$ consists of all block-lower-triangular matrices, $\U_W := \exp(\ug_W)$, and 
\[
  \G_W := Z(W) = \{ h\in \Gl_n(\RR) : hW = W h \}, \qquad \Q_W := \G_W \U_W, \qquad \K_W :=  \Q_W \cap \Or(n).
\]

\begin{lemma}
The linear functional 
\[
    \tr_W : \qg_W \to \RR, \qquad A \mapsto  \tr_W(A) := \tr(A W),
\] 
is a Lie algebra homomorphism.
\end{lemma}
\begin{proof}
Recall that $\qg_W = \ggo_W \oplus \ug_W$,  $[\qg_W, \qg_W] = [\ggo_W, \ggo_W] \oplus \ug_W$. Since the elements of $\ug_W$ are strictly lower triangular matrices,  $\tr \ug_W W = 0$. On the other hand, on $\ggo_W$ the result follows by applying, on each of the eigenspaces of $W$, the fact that the commutators are traceless.
\end{proof}

\begin{proposition}\label{prop_deta}
There exists a unique Lie group homomorphism $\det_W : \Q_W \to \RR^+$  whose differential at the identity is $\tr_W : \qg_W \to \RR$.
\end{proposition}

\begin{proof}
Let $V_1 \oplus \cdots \oplus V_l$ be the eigenspace decomposition of $W$, corresponding respectively to the different eigenvalues $w_1 < \cdots < w_l$. For $q = hu \in \Q_W$, $h\in \G_W$, $u\in \U_W$, we define 
\[
    {\det}_W (q) := \prod_{i=1}^l | {\det h_i} |^{w_i},
\]
where $h_i := h|_{V_i}$. A straightforward computation shows that its differential at the identity is indeed $\tr_W$. To see that it is a  group morphism, notice that the decomposition $\Q_W = \G_W \U_W$ induces a Lie group isomorphism 
\[
    \Q_W/ \U_W \simeq \G_W \simeq \Gl(V_1) \times \cdots \times \Gl(V_l).
\] 
Then $\det_W$ is nothing but the composition of these isomorphisms with taking the product of the absolute value of the usual determinant function (raised to the corresponding power) on each factor $\Gl(V_i)$.
\end{proof}

\begin{definition}
We call $\det_W$ the \emph{weighted determinant} function with weights $W \in \glg_n(\RR)$.
\end{definition}

Notice that for any $k\in \K_W \leq \G_W$ one has $\det_W(k) = 1$.

\begin{remark}
If we set $W = \Id$ then we recover the usual determinant function. On the other hand, for a generic $W$ with $n$ different eigenvalues $w_1 < \cdots < w_n$, we have
\[
    {\det}_W : \Q_W =  \{ q\in \Gl_n(\RR)  : q_{ij} = 0 \, \hbox{if } j>i \} \to \RR^+, \qquad {\det}_W (q) = \prod_{i=1}^n |q_{ii}|^{w_i}.
  \] 
\end{remark}

\subsection{The weighted volume density}\label{sec_beta_vol_dens}

We now extend the weighted determinant function to bilinear forms.  Let $\ngo$ be a real vector space and consider the set  $\Sym^2_{\geq 0}(\ngo^*)$ of  positive semi-definite scalar products on $\ngo$.
Let  us fix a background, positive-definite scalar product $\bkg$ on $\ngo$, a $\bkg$-orthonormal basis $\{ \bar e_i\}_{i=1}^d$, and a set of weights $W = (w_1, \ldots, w_d)$ with $w_1\leq \cdots \leq w_d$. Any positive definite $h \in \Sym^2_{+}(\ngo^*)$  can be written as
\begin{equation}\label{eqn_qcdoth}
    h = q \cdot \bkg := \bkg(q^{-1} \cdot, q^{-1} \cdot), \qquad q \in \Q_W.
\end{equation}
It is not hard to see that the association $h\mapsto q$ can be chosen so that it defines  a smooth map 
\begin{equation}\label{eqn_sym2intoQ}
        \Sym^2_+(\ngo^*) \to \Q_W.
\end{equation}
Indeed,  with respect to the fixed basis $\{\bar e_i\}$, we may take $q$ to be lower triangular and with positive diagonal entries (see \cite[$\S$5]{alek_sol}).

\begin{definition}\label{def_vbp}
The  weighted volume density of $h \in \Sym^2_{+}(\ngo^*)$ is defined as 
\[
     v_W (h) :=  \begin{cases}
            {\det}_{-W}(q),  \qquad &\hbox{if } h = q \cdot \bkg \hbox{ is positive definite;}\\
            0, \qquad &\hbox{otherwise}.
            \end{cases}
\]
\end{definition}

The minus sign in $-W$ and the absence of a square root can be explained by the fact that $q$ scales inverse quadratically with respect to $h$. We notice that $v_W$ is well-defined: indeed, if $h = q_1 \cdot \bkg = q_2 \cdot \bkg$ then $q_2^{-1} q_1 \in \Or(\ngo,\bkg) \cap \Q_W = \K_W$, thus $\det_{-W}(q_1) = \det_{-W}(q_2)$.


\begin{proposition}\label{prop_vbp}
The weighted volume density is well-defined and smooth on positive-definite inner products. Moreover, if $W>0$  then $v_W$ can be continuously extended to $\Sym^2_{\geq 0}(\ngo^*)$. 
\end{proposition}

\begin{proof}
If $ h = q \cdot \bkg =  q' \cdot \bkg$ then $q' = q k$, $k\in  \Q_W \cap \Or(\ngo, \bkg) = \K_W$. Since $\det_{-W}$ is multiplicative, we have $\det_{-W}(q') = \det_{-W}(q) \det_{-W}(k)$, with $\det_{-W}(k) = 1$ because the restriction of $k$ to each eigenspace of $W$ is orthogonal. Hence, $v_W$ is well defined.

Finally, thanks to smoothness of \eqref{eqn_sym2intoQ} and of $\det_{-W}$, $v_W$ is smooth on the open and dense subset of positive definite metrics. Thus, to extend continuously to all of $\Sym^2_{\geq 0}(\ngo^*)$  we must show that $v_W(h_k) \to 0$ for any sequence $(h_k)_{k\in \NN}$ in $\Sym^2_+(\ngo^*)$ with $h_k \to h_\infty  \in \Sym^2_{\geq 0}(\ngo^*) \setminus \Sym^2_+(\ngo^*)$, provided $W>0$. 

To that end, write $h_k = q_k \cdot \bkg$, with $q_k = \exp(E^{(k)})$, $E^{(k)} \in \qg_W$, lower-triangular with respect to $\{\bar e_i\}$, and with diagonal entries $E^{(k)}_{ii} \in \RR$, $i=1, \ldots, \dim \ngo$. Since
\[
    h_k(\bar e_i, \bar e_i) = \Vert \exp(-E^{(k)}) \bar e_i \Vert_{\bkg}^2  = e^{-2 E^{(k)}_{ii}} + \hbox{ non-negative terms}
\] 
converges to $h_\infty(\bar e_i,\bar e_i)$ as $k\to\infty$, the sequences of eigenvalues $\{ E^{(k)}_{ii}\}_{k\in \NN}$ are uniformly bounded below. On the other hand, if along a subsequence all the eigenvalues are bounded above, then the limit metric $h_\infty$ would be positive definite. Hence 
\[
    \max\left\{E^{(k)}_{ii} : 1\leq i \leq \dim \ngo \right\} \overset{k\to \infty}\longrightarrow +\infty.
\]
Since $w_i > 0$ for all $i$, the above discussion yields that $\tr_W E_k \to \infty$ as $k\to\infty$. Therefore, 
\[
        v_W(h_k) = {\det}_{-W} (q_k) = \exp(-\tr E^{(k)} W)  \to 0.
\]
\end{proof}

\begin{example}
When $\ngo = \Lie(\N)$ we have that $v_{\Id} = v_\N$. Indeed, if $\{U_i^*\}$ is $\bkg$-orthonormal and $h = q\cdot \bkg$ denotes the inner product induced by $g$ on $\ngo$ via evaluation of Killing fields at $p$, then 
\[
    g_{ij}(p) = \bkg(q^{-1} U_i,  q^{-1} U_j) = (q q^T)^{-1}_{ij},
\]
thus  $(\det g_{ij}(p))^{1/2} = (\det q)^{-1}$.
\end{example}

\subsection{The $\beta^+$-weighted volume density}\label{sec_weig_vol_dens}

Let $\N$ be a connected nilpotent Lie group with Lie algebra $\ngo$. After fixing a background left-invariant metric on $\N$, defined by an inner product $\bkg \in \Sym^2_+(\ngo^*)$, consider the associated $\bkg$-self-adjoint \emph{stratum label} $\beta \in \glg(\ngo)$, chosen so that the Lie bracket $\lb_\ngo$ of $\ngo$ is gauged correctly, see \cite[Appendix D]{alek_sol}.  We also set 
\[
    \beta^+ :=  (\tr \beta^2)^{-1} \, \beta + \Id_\ngo,
\] 
which is known to be positive definite since $\ngo$ is nilpotent \cite{standard}. Applying \Cref{def_vbp} to the set of weights $W$ given by the eigenvalues of $\beta^+$ in non-decreasing order we obtain the \emph{$\beta^+$-weighted volume density}
\[
    \vbp : \Sym^2_{\geq 0}(\ngo^*) \to \RR.
\]

\begin{example}\label{ex_heis}
Let $\ngo$ be the $3$-dimensional Heisenberg Lie algebra, spanned by a basis $\{e_1,e_2,e_3\}$ where the only non-zero bracket is $[e_1,e_2] = e_3$. Let $\bkg$ make that basis orthonormal. The stratum label is $\beta = \diag(-1,-1,1)$ \cite{Lau2003} gives the correct gauge for the Lie bracket (this follows from the fact that $\beta^+ \in \Der(\ngo)$), and we have 
\[
  \beta^+ = \diag \left(\tfrac23, \tfrac23, \tfrac43 \right).
\]
For any other inner product $h$ on $\ngo$ we may perform a Gram-Schmidt orthogonalization process on the basis $\{e_i\}$, starting from $e_3$ and going backwards. In this way it is not hard to see that the entries of the matrix $q$ relating $h$ with $\bkg$ via $h = \bkg(q^{-1} \cdot, q^{-1} \cdot)$ satisfies
\begin{align*}
  q_{33} &= h(e_3,e_3)^{-1/2}, \\
  q_{22} &= h(u_2,u_2)^{-1/2}, \qquad u_2 = e_2 - \frac{h_{23}}{h_{33}} e_3, \\
  q_{11} &= h(u_1,u_1)^{-1/2}, \qquad u_1 = e_1 - \frac{h_{12}}{h_{22}} e_2 - \frac{h_{13}}{h_{33}} e_3.
\end{align*}
Here $h_{ij} := h(e_i,e_j)$. Finally, the $\beta^+$-weighted volume density is given by
\[
  v_{\beta^+}(h) = \left(h(u_1,u_1) \, h(u_2,u_2) \, h(e_3, e_3)^{2}\right)^{1/3}.
\]
In the particular case of a \emph{diagonal} metric $h$, i.e. one for which $\{e_i\}$ is orthogonal, we have
\[
  v_{\beta^+}(h) = \left( h_{11} \, h_{22} \, h_{33}^2 \right)^{1/3}.
\]
\end{example}

\begin{lemma}\label{lem_vbpautinv}
If for some positive-definite inner products $h_1, h_2$ there exists $\varphi \in \Aut(\ngo)$ defining a linear isometry $\varphi : (\ngo, h_1) \to (\ngo, h_2)$, then 
\[
        \vbp(h_2) = \det(\varphi)^{-1} \, \vbp(h_1).
\]
In particular, $\vbp$ is $(\Aut(\ngo)\cap \Sl(\ngo))$-invariant.
\end{lemma}

\begin{proof}
If $h_1 = q_1 \cdot \bkg$ then by assumption we have $h_2 = (\varphi q_1) \cdot \bkg$. Thus, by definition of $\vbp$ we have
\[
        \vbp(h_2) = {\det}_{-\beta^+} (\varphi q_1) = {\det}_{-\beta^+} (\varphi) \vbp(h_1).
\]
Also, 
\[
   {\det}_{-\beta^+} (\varphi) =  \left({\det}_{\beta} (\varphi) \right)^{-1/\tr \beta^2} \, {\det}_{-\Id}(\varphi) = \det(\varphi)^{-1},
\]
since $\Aut(\ngo) \subset \det_\beta^{-1}(1)$ \cite[$\S$7]{BL17}.
\end{proof}

\subsection{The weighted volume density of the orbits}\label{sec_hvbp}

In this section we finally define the weighted volume density that will play a key role in the proof of Theorem \ref{thm_unimod}.

Consider a nilpotent Lie group $\N$ acting properly and isometrically on a complete Riemannian manifold $(M,g)$, which we assume is also endowed with a semi-definite metric $g_s$.  Evaluating Killing fields, $g_s$ gives rise to a  smooth map 
\[
    h : M \to \Sym^2_{\geq 0}(\ngo^*), \qquad p\mapsto h_p, \qquad h_p(U,V) :=  g_s(U^*_p, V^*_p), \qquad U,V\in \ngo,
\]
associating  to each $p\in M$ a semi-definite inner product in $\ngo$. 
We define the $\beta^+$-weighted volume density of the $\N$-orbits with respect to the semi-definite metric $g_s$,  as the smooth function 
\begin{equation}\label{eqn_vbpdef}
  \vbp : M \to \RR, \qquad  \vbp(p) := \vbp(h_p),
\end{equation}
where the right-hand-side was defined in $\S$\ref{sec_weig_vol_dens}.

\begin{lemma}\label{lem_vbp_finv}
Let $f\in \Isom(M,g)$ normalise $\N$ (as per the discussion preceding Lemma \ref{lem_vN-finv}). Conjugation by $f$ in $\N$ induces a Lie algebra automorphism denoted by $\Ad(f)|_\ngo \in \Aut(\ngo)$. Then,
\[
    \vbp(f(p)) = \det \left( \Ad(f)|_\ngo \right)^{-1} \, \vbp(p).
\]
In particular, $\vbp$ is $\N$-invariant.
\end{lemma}

\begin{proof}
The inner products $h_{f(p)}, h_p$ induced by $ g_s$ on $\ngo$ at $f(p)$ and $p$, respectively, are such that
\[
    \Ad(f)|_\ngo : \left(\ngo, h_p\right) \to \big(\ngo, h_{f(p)}\big)
\] 
is a linear isometry and an automorphism.  The formula thus follows from \Cref{lem_vbpautinv}. The last claim  follows by applying the formula to $f\in \N$, and the fact that $\N$ is unimodular. 
\end{proof}

The above lemma is absolutely key for geometric applications, as it will imply the $\G$-invariance of the function $\vbp: M \to \RR$ provided $\N\lhd \G$ is the nilradical of a unimodular Lie group $\G$.

For the reader's convenience, we connect the above definition \eqref{eqn_vbpdef} with the one used in \cite{alek_sol}, and recall a useful estimate. To that end, let $M_{\sf reg}$ denote the regular part for the $\N$-action, and set $P := \N\backslash M_{\sf reg}$. By \Cref{lem_vbp_finv} we may write $\vbp = \vbpu \circ \pi^\N$, for some smooth function $\vbpu \in \cca^\infty(P)$.

\begin{lemma}\label{lem_vbp=vbvn}
Let $\vb$ be defined as in \cite[$\S$5]{alek_sol}. Then, $\vbp = \vb  v_\N$, and on $P$ we have 
\[
   \Delta_P  \, \log \vbpu \geq \sum_{j=1}^d \tr \left( \left( \nabla_{X_j} L\right)_{X_j} \cdot \beta^+ \right),
\]
where $\{ X_j\}$ is any local orthonormal frame for $\ho^\N$.
\end{lemma}

\begin{proof}
As functions on $\Sym^2_+(\ngo^*)$, both $\vbp$ and $\vb v_\N$ evaluate to $1$ on the background inner product $\bkg$, and their logarithmic first variations are both given by
\begin{equation}\label{eqn_gradlogvbp}
    X(\log \vbp) = X(\log  v_\beta + \log  
    v_\N) =  \la L_X, \beta^+ \ra, \qquad X\in \ho^\N,
\end{equation}
thanks to \cite[Lemma 5.9]{alek_sol} and \eqref{eqn_mcv_logvn}. Thus, $\vbp = \vb v_\N$. 

Regarding the estimate, we have $\vbpu = \underline{v}_\beta \underline{v}_\N$, where $\underline{v}_\beta \circ \pi^\N = \vb, \underline{v}_\N \circ \pi^\N = v_\N$. By \cite[Lemma 2.3]{alek_sol} and \eqref{eqn_mcv_logvn} we have 
\[
  \Delta_P \log \underline{v}_\N = - \divg_P (\pi^\N_*\mcv^\N) = \sum_{j=1}^d \tr \left( (\nabla_{X_j} L)_{X_j} \right).
\]
On the other hand, \cite[Lemma 5.9]{alek_sol} gives 
\[
  \Delta_P \, \log \underline{v}_\beta \geq \sum_{j=1}^d \tr \left( (\nabla_{X_j} L )_{X_j} \cdot (\beta^+ - \Id_{\ve^\N} \right),
\]
and the claim follows from the fact that $\Delta_P \log \vbpu = \Delta_P \log \underline{v}_\beta + \Delta_P \log \underline{v}_\N$. (Notice that the lemmas from \cite{alek_sol} are concerned with local curvature estimates, and thus can be applied in this more general situation provided we are in the regular part.)
\end{proof}

\section{Negative Ricci curvature and unimodular symmetry}\label{sec_ricneg}

In this section our aim is to prove \Cref{thm_ricneg}, which is the first step towards \Cref{thm_unimod}, and holds under the more general (partial) negative Ricci assumption.

\begin{theorem}\label{thm_ricneg}
Let $\G$ be a connected unimodular Lie group acting cocompactly, properly and isometrically on a  Riemannian manifold $(M^n, g)$. Assume that the Ricci curvature of $g$ is strictly negative in directions tangent to the $\G$-orbits. Then, $\G$ is semisimple.
\end{theorem}

Arguing by contradiction, let $(M^n,g)$ and $\G$ be as in the theorem's statement, and  assume that  $\G$ is not semisimple.  Then its nilradical $\N$ has positive dimension. Consider the lifted action of $\N$ on the frame bundle $FM$, which by the discussion in $\S$\ref{sec_liftiso}, is isometric and free. We endow $FM$ with the positive semi-definite metric $\hat g_s := \pi^* g$ which vanishes precisely on directions which are tangent to $\Or(n)$-orbits.  

We now apply $\S$\ref{sec_hvbp} to the $\N$-action on $FM$, endowed with the semi-definite metric $\hat g_s$. That is, after fixing a basis $\{e_i\}$ for $\ngo$, we consider the $\beta^+$-weigthed volume density $\vbph : FM \to \RR_{\geq 0}$ of the $\N$-orbits in $FM$ with respect to $\hat g_s$, as defined in $\S$\ref{sec_hvbp}. This is a continuous function which vanishes precisely on those frames $u\in FM_{\sf sing}$ for which $\N \cdot (\pi(u)) \subset M$ is a singular orbit.

\begin{lemma}\label{lem_hvbp}
The function $\hvbp : FM_{\sf reg} \to \RR_{>0}$ is smooth, $\G$-invariant and $\Or(n)$-invariant.
\end{lemma}

\begin{proof}
Smoothness follows from \Cref{prop_vbp} and the definition of $\hvbp$.

Since the nilradical is normal, any isometry $f\in \G$ normalises $\N$. Moreover, by unimodularity of $\G$ it follows that $\det \Ad(f)|_\ngo = 1$ (see \cite[Lemma 2.6]{alek}). Hence, $\hvbp$ is $\G$-invariant by \Cref{lem_vbp_finv}.

Regarding $\Or(n)$-invariance, let $r\in \Or(n)$ and denote by $R_r \in \Isom(FM, \hat g_s)$, $u\mapsto ur$, the corresponding isometry. Since the actions of $\N$ and $\Or(n)$ commute, the corresponding subgroups of $\Isom(FM,\hat g_s)$ commute as well. Hence, in the notation of \Cref{lem_vbp_finv} we have $\Ad(R_r)|_\ngo = \Id_\ngo$ and thus $\hvbp(ur) = \hvbp(u)$. 
\end{proof}

By $\Or(n)$-invariance, $\log \hvbp : FM_{\sf reg} \to \RR$ induces a smooth function $\log \vbp: M_{\sf reg} \to \RR$, which satisfies the following estimate:

\begin{lemma}\label{lem_estimate_Gunimod}
On the regular part $M_{\sf reg}$, $\log \vbp$ satisfies the differential inequality
\[
       \Delta_M \log \vbp + \la \Ric_M|_{\ve^\N}, \beta^+ \ra \geq 0.
\]
\end{lemma}
\begin{proof}
\Cref{lem_vbp=vbvn} gives that $\vbp = v_\beta v_\N$ on $M_{\sf reg}$. Thus, using the Riemannian submersion 
\[
    \pi^\N :    M_{\sf reg}  \to \N\backslash M_{\sf reg} =: P,
\] 
 we get
\begin{align*}
    \Delta_M \log \vbp  &= \Delta_{P} \log \vbpu -  g^P\left(\nabla \log \vbpu, \,  \pi^\N_* \, \mcv^\N \right) \\
    &\geq \sum_{j=1}^d \tr \left( \left( \nabla_{X_j} L\right)_{X_j} \cdot \beta^+ \right) -  \tr(\beta^+ L_{\mcv^\N})   \\
    &=  \la \Ric^{\N} + A^*A,  \beta^+ \ra - \la \Ric_M|_{\ve^\N}, \beta^+ \ra \\
    &\geq - \la \Ric_M|_{\ve^\N}, \beta^+ \ra.
\end{align*}
Here  $\vbpu \circ \pi^\N = \vbp$, $\underline{v}_\beta \circ \pi^\N = \vb$, $\underline{v}_\N \circ \pi^\N = v_\N$. The first equality is simply the formula for the divergence of invariant vector fields under an isometric group action \cite[Lemma B.3]{alek_sol}; the estimate in the second line uses \Cref{lem_vbp=vbvn} and \eqref{eqn_gradlogvbp}; the third line is just \eqref{eqn_EV}; 
and  the last estimate follows from  $A^*A \geq 0$, $\beta^+ > 0$, and the GIT Ricci curvature estimate $\la \Ric^\N, \beta^+\ra \geq 0$. 
\end{proof}

\begin{proof}[Proof of \Cref{thm_ricneg}]
Assuming that the nilradical of $\G$ has positive dimension, we may define the smooth function $\log \vbph : FM_{\sf reg} \to \RR$ as above, which by construction approaches $-\infty$ as we approach $FM_{\sf sing}$. Since $\G \backslash FM$ is compact and $\log \vbph$ is $\G$-invariant (\Cref{lem_hvbp}), $\log \vbph$ attains a global maximum somewhere in the non-singular part $FM_{\sf reg}$. By $\Or(n)$-invariance, this means that $\log \vbp$ also attains a maximum on $M_{\sf reg}$. Using \Cref{lem_estimate_Gunimod} at that point and the fact that $\beta^+ > 0$ and $\Ric_M < 0$ along orbits, we obtain a contradiction. Hence the nilradical must be trivial and $\G$ is semisimple.
\end{proof}

\section{Einstein manifolds with semisimple symmetry}\label{sec_ss}

In this section we continue working towards a proof  of Theorem \ref{thm_unimod}. The main result is \Cref{prop_Lpolar}. From now on and for the rest of the paper, unless otherwise stated, we assume that  $(M^n,g)$ is a complete Riemannian manifold with $\ric_g = -g$, on which a connected unimodular Lie group  $\Ll$ acts properly, isometrically and cocompactly. After taking a quotienti by the ineffective kernel ---a normal subgroup---, we may and will also assume that the action is effective.  By Theorem \ref{thm_ricneg}, we know that  $\Ll$ is semisimple.  Moreover, if we write $\Ll = \Ll_{c} \Ll_{nc}$, where $\Ll_c$ (resp.~ $\Ll_{nc}$) is the product of all normal compact (resp.~non-compact) simple factors, then the action of $\Ll_{nc}$ on $(M^n,g)$ is also proper, isometric and cocompact. Thus, we may assume that $\Ll = \Ll_{nc}$.

\begin{lemma}\label{lem_centerless}
The center $\Zz$ of $\Ll$ has a finite-index subgroup $\Zz_{f}\leq \Zz$ acting freely on $(M^n,g)$. In particular, $M/{\Zz_f}$ is locally isometric to $M$ and admits a proper, isometric and cocompact action of the semisimple Lie group $\Ll/{\Zz_f}$, which has finite center. 
\end{lemma}

\begin{proof}
By an old result of P.A. Smith \cite{Smi35} (see \Cref{prop_fin_gen}), $\Zz$ is a finitely-generated abelian group.  Let $\Zz = \Zz_f \Zz_t$ be the decomposition into its free abelian part $\Zz_f$ and the subgroup consisting of torsion elements $\Zz_t$. We  claim that $\Zz_f \leq \Ll$ acts freely on $M^n$. To see that, recall that for each $p\in M$ the isotropy group $\Ll_p$ is compact by properness. Hence, $\Zz_f \cap \Ll_p$ is compact and discrete, thus finite. Since $\Zz_f$ does not contain any non-trivial finite subgroups, it follows that $\Zz_f\cap \Ll_p = 1$, that is, $\Zz_f$ acts freely. Clearly, the quotient group $\Ll/\Zz_f$ acts freely, isometrically and cocompactly on $M/\Zz_f$, the latter endowed with the unique metric so that the covering map $M \to M/\Zz_f$ is a local isometry. Finally, the center of $\Ll/\Zz_f$ is isomorphic to $\Zz_t$, and is in particular finite. 
\end{proof}

In summary, throughout  the following sections we may and will assume that
\begin{assumption}
$\Ll$ is a connected, semisimple Lie group without compact simple factors and with finite center, acting properly, isometrically and cocompactly on a complete Einstein manifold $(M^n,g)$ with $\ric_g = -g$.
\end{assumption}

Consider now any Iwasawa decomposition
\[
    \Ll = \K \A \N,
\] 
with corresponding decomposition  $\lgo = \kg \oplus \ag \oplus \ngo$ on Lie algebra level. Recall that the latter is obtained from a Cartan involution $\theta \in \Aut(\lgo)$ as follows: the $+1$ and $-1$ eigenspaces for $\theta$ give rise to a  Cartan decomposition $\lgo = \kg \oplus \pg$. One then takes  a maximal abelian subspace $\ag \subset \pg$, and after choosing a notion of positivity for the restricted roots in $\ag^*$, $\ngo$ is the direct sum of all positive root spaces. 

Since $\Ll$ has finite center, $\K$ is a maximal compact subgroup of $\Ll$ \cite[Thm.~6.31]{Knapp2e}. The Lie subgroup $\G := \A \N \leq \Ll$ is solvable, simply-connected, and we have a diffeomorphism $\Ll = \K \G  \simeq \K \times \G $; in particular $\G$ is closed in $\Ll $ and $\G \cap \K = 1$. By analogy with the complex case, we call $\G$ a \emph{Borel subgroup}. (More precisely, $\G$ is simply the radical of a minimal parabolic subgroup.)  Let $\lgo = \kg \oplus \ag \oplus \ngo$ be the corresponding decomposition on Lie algebra level.  We say $\ggo := \Lie(\G) = \ag \oplus \ngo$ is a Borel subalgebra of $\lgo$, and $\ngo$ (and also $\N$) a \emph{Borel nilradical}. 

We denote by $\ngo^- := \theta(\ngo)$ the Borel nilradical corresponding to the opposite choice of positive roots. Recall that $\lgo = \ngo^- \oplus \mg \oplus \ag \oplus \ngo$, where $\mg = Z_\kg(\ag)$. When $\mg = 0$ the Lie algebra $\lgo$ is called split-semisimple (e.g. $\lgo = \slg(n,\RR)$), however this is far from true in general (e.g. for $\lgo = \sog(2,n)$, $\mg \simeq \sog(n-1)$).

\begin{lemma}\label{lem_GMcocompact}
The Borel subgroup $\G \leq \Ll$ acts on $(M,g)$ freely, properly, isometrically and cocompactly.
\end{lemma}

\begin{proof}
Since $\Ll$ acts effectively and isometrically, the same is true for $\G$. Also, $\G\leq \Ll$ is  closed, thus properness of the $\Ll$-action implies that of $\G$ by \Cref{lem_proper}. Moreover, $\G$ is solvable and simply-connected, hence  it has no nontrivial compact subgroups.  It follows that $\G$ must act freely, and that $\G\backslash M$ is a manifold. 

Consider now the continuous and surjective map 
\[
    \G \backslash M \to \Ll \backslash M, \qquad  \G \cdot p \mapsto \Ll \cdot p.
\]
We claim that its fibers are the compact sets $\{ \G \cdot (k\cdot p) : k\in K \} \subset \G\backslash M$. Indeed, since $\Ll = \K\G$,  two $\Ll$-orbits $\Ll\cdot p$, $\Ll\cdot q$ coincide if and only if $p = ks \cdot q$ for some $k\in \K, s\in \G$, and the latter is in turn equivalent to $\G \cdot (k^{-1} \cdot p) = \G \cdot q$.  
Compactness of $\G\backslash M$ now easily follows from this claim and the compactness of $\Ll \backslash M$.
\end{proof}

Combining this with the results from \cite{alek_sol} and a Lie theoretic property of semisimple Lie algebras (\Cref{app_sslie}), we deduce the following:

\begin{proposition}\label{prop_Lpolar}
Let $(M^n, g)$, $\Ll$ be as in $\S$\ref{sec_ss}. Then, the action of $\Ll$ is polar. 
\end{proposition}

\begin{proof}
As remarked above, it suffices to show that 
\begin{equation}\label{eqn_A=0}
    \la \nabla_X U^*, Y \ra = 0,
\end{equation}
for all $X,Y \in \Gamma(\ho^\Ll)$ and all $U\in \lgo$.   By \Cref{lem_GMcocompact} we may apply \Cref{thm_Naction}, \ref{item_Npolar} and deduce the vanishing of the integrability tensor $A^\N$ of the Riemannian submersion $M \to \N\backslash M$. In other words, \eqref{eqn_A=0}  holds for all $X,Y \in \Gamma(\ho^\N)$ and all $U\in \ngo$. Since $\ho^\Ll \subset \ho^\N$,  \eqref{eqn_A=0} holds for all $X,Y\in \Gamma(\ho^\Ll)$ and  all $U \in \ngo$. 

By varying the Iwasawa decomposition and applying \Cref{thm_Naction}, \ref{item_Npolar} to the action of other Borel subgroups $\G$, we obtain that \eqref{eqn_A=0} holds for all $U$ in \emph{some} Borel nilradical of $\lgo$. Since \eqref{eqn_A=0} is linear in $U$, the proposition follows from the fact that the Borel nilradicals linearly span all of $\lgo$ (\Cref{prop_nilspan}).
\end{proof}

From \Cref{lem_GMcocompact} we also obtain all  the structure results for the Riemannian submersion $M \to \N\backslash M$ arising  from \cite{alek_sol} and listed in \Cref{thm_Naction}.

\section{The Lie-theoretic structure of the orbits}\label{sec_Qtrans}

In this section our first goal is to prove \Cref{prop_Qtransitive}, which is key for our purposes, and to discuss a number of very important consequences of it. Recall that if $\M := N_\K(\ag)$ denotes the normaliser of $\ag$ in $\K$, one defines the (identity component of the) \emph{minimal parabolic subgroup} via
\begin{equation}\label{eqn_Q}
    \Q := (\M \A \N)_0.
\end{equation}
For details about these subgroups see \cite[Ch.VII, $\S$7]{Knapp2e}.

\begin{proposition}\label{prop_Qtransitive}
Let $(M^n,g), \Ll$ be as in $\S$\ref{sec_ss} and let $\Q$ be defined as in \eqref{eqn_Q}, for any choice of Iwasawa decomposition. Then, $\Q$ acts transitively on all $\Ll$-orbits: 
that is,  $\Q \Ll_p = \Ll$ for any  $p\in M$.
\end{proposition}

We introduce the following `partial trace' of a vertical endomorphism $E\in \End(\ve^\Ll)$:
\[
    \tr_{\ve^\N} E (p) := \sum_r g(E U_r, U_r)_p. 
\]
Here, $\{U_r\}_{r=1}^{\dim \ngo}$ is any vertical frame for $\ve^\N$ which is orthonormal at $p$ (it is easy to see that the definition is independent of such frame).
The proof of \Cref{prop_Qtransitive} will also yield:

\begin{corollary}\label{cor_tr_NRicL}
If $p\in M$ satisfies the condition \eqref{eqn_aperpn} below,  then  $\tr_{\ve^\N} \Ric_p^{\Ll} = -\dim \ngo$.
\end{corollary}

We now start working towards proving \Cref{prop_Qtransitive}. Consider any identity component of a minimal parabolic $\Q = \M \A \N \leq \Ll$ associated to an Iwasawa decomposition $\Ll = \K \A \N$, and let $\Ll \cdot  p' \subset M$ be a principal $\Ll$-orbit.  Recall  that, due to the isometric action of $\Ll$ on $M$, we may identify any $X\in \lgo$ with the Killing field $X^* \in \Gamma(TM)$, and for  $\sg \subset \lgo$ we set 
\[
    \sg \cdot p :=   \{X^*_p : X \in \sg \} \subset \ve^\Ll(p) = T_p (\Ll\cdot p) \subset T_p M.
\]
We now show that we can change the point $ p'$  to another one on the same $\Ll$-orbit at which the algebra and the geometry are compatible:

\begin{lemma}\label{lem_aperpn}
There exists $p\in \N\cdot  p' \subset \Ll \cdot  p'$ such that
\begin{equation}\label{eqn_aperpn}
    \ag \cdot p \perp \ngo \cdot p.
\end{equation}
Moreover, if $\sigma  : \Sigma_\N \to M$ is a section through $p$ for the polar action of $\N$, then \eqref{eqn_aperpn} also holds along $\sigma(\Sigma_\N)$.
\end{lemma}

\begin{proof}
Polarity of the $\N$-action (\Cref{thm_Naction}, \ref{item_Npolar}) implies that the simply-connected solvmanifold $\G \cdot  p'$ is \emph{standard}, in the sense of \cite{Heb}: that is, the subspace of Killing fields 
\[
        \ag_{ p'} := \{ X\in \ggo : X^*_{ p'} \perp \ngo \cdot  p' \}
\]
is an abelian Lie subalgebra of $\ggo$ complementary to $\ngo$. By \cite[Cor.~2.10]{Heb} (which may be used since $\G$ comes from an Iwasawa decomposition), $\Ad(\N)$ acts transitively on the set of such abelian complements, thus $\ag = \Ad_{x}(\ag_{p'})$ for some $x\in \N$. Then, for $p= x\cdot p'$, we have $\ag_p = \Ad_x(\ag_{p'}) = \ag$, as desired.

For the second claim, let $a\in \A$, let $L_a : M \to M$ denote the corresponding isometry, and set $q := a\cdot p$. Since $\A$ normalises $\N$, $a \cdot (\N\cdot p) = (a \N a^{-1}) \cdot (a \cdot p) = \N \cdot q$, thus $\A$ maps $\N$-orbits to $\N$-orbits. In particular, $(L_a)_*$ maps $\ve^\N(p)$ to $\ve^\N(q)$ and $\ho^\N(p)$ to $\ho^\N(q)$. Since $\ag\cdot p \subset \ho^\N(p)$ and $(L_a)_*(\ag\cdot p) = \ag \cdot q$, we deduce that $\ag \cdot q \subset \ho^\N(q)$. This implies that $\A \cdot p \subset \sigma(\Sigma_\N)$.

We now claim that $\Sigma_\N$ is invariant under $\A$, from which the second claim follows immediately. Let $\gamma : I \to \sigma(\Sigma_\N)\subset M$ be a geodesic with $\gamma(0) = p$. For any $a\in \A$, $a\cdot \gamma$ is a geodesic  through $q \in \sigma(\Sigma_\N)$, with velocity $ (a\cdot \gamma)'(0) = (L_a)_* (\gamma'(0)) \in \ho^\N(q)$.  Hence, $a \cdot \gamma$ is a horizontal geodesic. It follows that $\sigma(\Sigma_\N)$ is $\A$-invariant. 
\end{proof}

By \Cref{prop_Lpolar}, through each $p\in M$ there exists a \emph{section} $\sigma:\Sigma_\Ll \to M$ for the $\Ll$-action: an (immersed) submanifold intersecting all $\Ll$-orbits orthogonally \cite{HLO,GZ12}. It is well-known that the section is a totally geodesic submanifold, hence, by \Cref{lem_aperpn}, \Cref{eqn_aperpn} is satisfied for all points on $\sigma(\Sigma_\Ll)$, provided it holds for $p$. We endow $\Sigma_\Ll$ with the Riemannian metric $\sigma^* g$.

Consider the smooth function
\[
    \log v_\N \circ \sigma : \Sigma_\Ll \to \RR,
\]
where $v_\N$ is the relative volume density of the $\N$-orbits in $(M^n,g)$ (see $\S$\ref{sec_volume_density}). To prove Proposition \ref{prop_Qtransitive}, we will establish a formula for the Laplacian of $\log v_\N \circ \sigma$ and an algebraic estimate for the zeroth-order term in said formula. The potential non-compactness of $\Sigma_\Ll$ plays no role here, since $\log v_\N \circ \sigma$ turns out to be constant by \Cref{thm_Naction}, \Cref{item_HN_Gvert}.

\begin{lemma}
The function $\log v_\N \circ \sigma : \Sigma_\Ll \to \RR$ satisfies the following PDE on $\Sigma_\Ll$:
\begin{equation}\label{eqn_Dellogvn}
            \Delta_{\Sigma_\Ll} (\log v_\N \circ\sigma) - \big\la \mcv^\Ll , \nabla (\log v_\N\circ\sigma) \big\ra = \dim \ngo + \tr_{\ve^\N} \Ric^{\Ll}.
\end{equation}
\end{lemma}

\begin{proof}
Since $\sigma$ is a local isometry and the computation is local, we may perform it on $\sigma(\Sigma_\Ll) \subset M$. Let $p\in \sigma(\Sigma_\Ll)$, let $\{U_r \}$ be a local orthonormal frame for $\ve^\N$, and choose a local orthonormal frame  $\{X_j\}$ for $\ho^\Ll$ such that $\nabla_{X_j} {X_j}\big|_p = 0$ for all $j$. We compute:
\begin{align*}
    \Delta_{\sigma(\Sigma_\Ll)} \log v_\N &=  -\sum_j \la \nabla_{X_j} \mcv^\N, X_j \ra = -\sum_{j,r}  X_j \la \nabla_{U_r} U_r, X_j \ra  = \sum_{j,r} X_j \la L_{X_j} U_r, U_r \ra \\
    &=  \sum_{j,r}  \left( \la (\nabla_{X_j} L)_{X_j} U_r, U_r \ra + \la L_{X_j} \nabla_{X_j} U_r, U_r \ra + \la L_{X_j} U_r, \nabla_{X_j} U_r \ra \right) \\
    &= \tr_{\ve^\N} \sum_j  (\nabla_{X_j} L)_{X_j} +  2 \, \sum_{j,r,s} \la L_{X_j} U_r, U_s \ra \, \la \nabla_{X_j} U_r, U_s \ra \\
    & = \tr_{\ve^\N} \sum_j  (\nabla_{X_j} L)_{X_j},
\end{align*}
where $L_{X_j}$ are the shape operators of the $\Ll$-orbits. Notice that in the first equality we should have used the $\ho^\Ll$-component of $\nabla \log v_\N = -\mcv^\N$, but we may instead compute with $-\mcv^\N$ thanks to $A^\Ll = 0$. In the second-to-last equality we have used that the $L_{X_j}$ preserve the distribution $\ve^\N$, because the $\N$-action is polar (\Cref{thm_Naction}). Regarding the last equality, the  sum over $j,r,s$ vanishes since the first factor is symmetric in $r,s$ while the second one is alternating (notice also that $\nabla_{X_j} U_r$ is $\Ll$-vertical due to \Cref{prop_Lpolar}). Also,
\[
   -\la \mcv^\Ll, \nabla \log v_\N \ra =  \la \mcv^\Ll, \mcv^\N\ra  =   \sum_r \la \mcv^\Ll, \nabla_{U_r} U_r \ra = -\tr_{\ve^\N} L_{\mcv^\Ll}.
\]

Finally, since $A^\Ll=0$ by \Cref{prop_Lpolar}, the vertical part of the Einstein  equation for the $\Ll$-submersion \eqref{eqn_EV} (in endomorphism form) yields
\[
    -\Id_{\ve^\Ll} = \Ric^{\Ll} + L_{\mcv^\Ll} - \sum_j (\nabla_{X_j} L)_{X_j}.
\]
 Tracing over $\ve^\N$ gives
\[
    \tr_{\ve^\N} \sum_j (\nabla_{X_j} L)_{X_j} - \tr_{\ve^\N} L_{\mcv^\Ll} = \dim \ngo + \tr_{\ve^\ngo} \Ric^{\Ll},
\]
and the lemma follows.
\end{proof}

 Regarding the right-hand-side in \eqref{eqn_Dellogvn}, we have the following key algebraic curvature estimate:

\begin{lemma}\label{lem_alg_est}
Let $p\in M_{\sf reg}$ satisfy \eqref{eqn_aperpn}. Then, 
\[
     \dim \ngo + \tr_{\ve^\N} \Ric^{\Ll}(p) \geq 0,
\]
with equality if and only if $\Ll \cdot p = \Q \cdot p$, where $\Q$ is defined in \eqref{eqn_Q}.
\end{lemma}

\begin{proof}
We apply \cite[Prop.~12.4]{alek_sol} to the orbit $\Ll\cdot p$. Notice that $\N \leq \Ll$ is nilpotently embedded and acts freely and polarly on $\Ll\cdot p$. Then, that result gives
\begin{equation}\label{eqn_Prop12.4}
  \scal^\N(p) - \tr_{\ve^\N} \Ric^\Ll(p) = \sum_i \la \nabla_{E^*_i} E^*_i, \mcv^\N \ra_p + \sum_{i,r} \la [U^*_r,[U^*_r, E^*_i]], E^*_i \ra_p,
\end{equation}
where $\{E_i\}_{i=1}^n$ is any set in $\lgo$ containing a basis $\{U_r\}_{r=1}^{\dim \ngo}$ for $\ngo$, and so that $\{E_i^* \}$ is an orthonormal basis of $T_p (\Ll\cdot p)$.  Choose now $A\in \lgo$ with $A^*_p = -\mcv^\N_p$. By \Cref{thm_Naction}, \Cref{item_HN_Gvert}, we may assume $A\in \ggo = \ag \oplus \ngo$. Condition \eqref{eqn_aperpn} then gives $A\in \ag$. By definition of the Iwasawa decomposition, this implies that $D:= \ad_\lgo A : \lgo \to \lgo$ is diagonalisable with real eigenvalues. Arguing precisely as in the proof of \cite[Lemma 14.2]{alek_sol}, it follows from \Cref{thm_Naction}, \eqref{eqn_RicveN} that $D|_{\ngo} > 0$, $D|_{\theta \ngo} < 0$, $D|_{\mg\oplus \ag} = 0$, and moreover $\tr D|_{\ngo} = \tr \beta^+$. Then, \cite[Prop.~12.5]{alek_sol} applies and gives a choice of Killing fields $\{E_i^*\}$ for which \eqref{eqn_Prop12.4} becomes
\[
  \scal^\N(p) - \tr_{\ve^\N} \Ric^\Ll(p) \leq \tr \beta^+,
\]
with equality if and only if $\Q \cdot p = \Ll \cdot p$. Replacing $\scal^\N$ using \eqref{eqn_RicveN} gives the desired estimate.
\end{proof}


\begin{proof}[Proof of Proposition \ref{prop_Qtransitive}]
Given a principal $\Ll$-orbit, by \Cref{lem_aperpn} we may choose a point $p\in M$ on it so that \eqref{eqn_aperpn} holds. By \Cref{lem_aperpn}, \eqref{eqn_aperpn} will also hold along the corresponding section $\Sigma_\Ll$ for the $\Ll$-action. Since $\mcv^\N = -\nabla \log v_\N$ is $\G$-vertical by \Cref{thm_Naction}, \eqref{item_HN_Gvert}, and in particular also $\Ll$-vertical,  $\log v_\N$  must be constant along $\Sigma_\Ll$. Thus, the left-hand-side in \eqref{eqn_Dellogvn} vanishes, giving equality in \Cref{lem_alg_est}, and hence $\Q\cdot p = \Ll \cdot p$. 
\end{proof}

\begin{proof}[Proof of Corollary \ref{cor_tr_NRicL}]
From the proof of Proposition \ref{prop_Qtransitive} we have equality in Lemma \ref{lem_alg_est}.
\end{proof}

We now present some remarkable  algebraic consequences of \Cref{prop_Qtransitive}.

\begin{corollary}\label{cor_kperpp}
At a point $p\in M_{\sf reg}$ satisfying \eqref{eqn_aperpn} the following hold:
\begin{enumerate}
  \item $\ngo \cdot p \perp ( \mg\oplus \ag) \cdot p$; \label{item_a+mperpn}
     \item  $\ngo \cdot p \perp \kg \cdot p$; \label{item_nperpk}
     \item $\ngo^- \cdot p \perp \kg \cdot p$; \label{item_n-perpk}
    \item $\ngo \cdot p \subset \pg \cdot p$. \label{item_npinpp}
\end{enumerate}
\end{corollary}

\begin{proof}
We first prove \Cref{item_a+mperpn}. Moreover, we claim that  
\[
    \mg \oplus \ag = \ngo^\perp := \{ X \in \qg : X^*_p \perp \ngo \cdot p \}.
\]
Notice that $\ngo^\perp \cap \ngo \subset \hg \cap \ngo = 0$, since the $\N$-action is free. Hence, $\dim \ngo^\perp = \dim \qg - \dim \ngo = \dim (\mg \oplus \ag)$ (recall that $\qg = \mg \oplus \ag \oplus \ngo$), and it suffices to prove that $\ngo^\perp \subset \mg\oplus \ag$. To that end, let $X = Y+U \in \ngo^\perp$, with $Y\in \mg\oplus \ag$, $U\in \ngo$. Using that $\ag \subset \ngo^\perp$ by \eqref{eqn_aperpn}, $\mg\oplus \ag = Z_\qg(\ag)$, and the fact that $\ngo^\perp$ is a Lie subalgebra of $\qg$ \cite[Cor.~8.7]{alek_sol}, for any $A\in \ag$ we obtain
\[
    \ngo^\perp \ni [X, A] = [U,A] \in [\ngo,\ag] \subset \ngo.
\]
From this it follows that $[X,A]^*_p = 0$, i.e.~ $[X,A] \in \hg \cap \ngo = 0$, thus $X\in \mg\oplus \ag$.

We now show \Cref{item_nperpk}. By \Cref{prop_Qtransitive} we know that $\Ll = \Q \Hh$, $\Hh = \Ll_p$. Since $\Ll = \K \G$ and $\Q = \M^0 \G$, it follows that 
\begin{equation}\label{eqn_K=MH}
    \K = \M^0 \Hh, \qquad \kg = \mg + \hg,
\end{equation}
 where $\hg$ is the isotropy subalgebra at $p$ ($\hg \cdot p = 0$).  Thus, $\kg \cdot p = \mg \cdot p \perp \ngo\cdot p$ by \Cref{item_a+mperpn}.

Regarding \Cref{item_n-perpk}, let $X\in \lgo_\lambda \subset \ngo$, for some positive root $\lambda\in \ag^*$. Thanks to \cite[Prop.~6.52(c)]{Knapp2e}, there exists $k\in N_\K(\ag)$ such that $\Ad(k) \lgo_\lambda = \lgo_{-\lambda}$.
By \eqref{eqn_K=MH}, we may write $k = h m$ for $h\in \Hh$, $m\in \M^0$. Using that $\Ad(m) \ngo \subset \ngo$, for all $Z\in \kg$ we have 
\[
        \la (\Ad(k) X)^*, Z^* \ra_p = \la (\Ad(m) X)^*,  (\Ad(h^{-1}) Z)^* \ra_p = 0,
\]
by $\Ad(\Hh)$-invariance of the metric and \Cref{item_nperpk}. Thus, $\lgo_{-\lambda} \cdot p\perp \kg \cdot p$ holds for all positive roots $\lambda$, and \Cref{item_n-perpk} follows. 

Finally, given $X\in \ngo$, we have $\theta X\in \ngo^-$ and  $X+\theta X\in \kg$. Items \ref{item_nperpk} and \ref{item_n-perpk} imply that  $(X+\theta X)^*_p \perp  X^*_p$, $(X+\theta X)^*_p \perp  (\theta X)^*_p$, thus $X+\theta X \in \hg$. Therefore,
\[
        X^*_p = \unm (X-\theta X)^*_p + \unm(X+\theta X)^*_p = \unm (X-\theta X)^*_p,
\]
and $\unm(X-\theta X) \in \pg$, from which the lemma follows.
\end{proof}

If $\lgo = \lgo_1 \oplus \cdots \oplus \lgo_r$  where each $\lgo_i$ is a simple ideal, and we set $\pg_i := \lgo_i \cap \pg$, then 
\[
        \pg = \pg_1 \oplus \cdots \oplus \pg_r
\]
is the decomposition of $\pg$ as a sum of irreducible, pairwise inequivalent, $\Ad(\K)$-modules.

\begin{corollary}\label{cor_AdHAdK}
The modules $\pg_i$, $i=1, \ldots, r$,  are irreducible and pairwise inequivalent as $\Ad(\Hh)$-modules.
\end{corollary}

\begin{proof}
To prove irreducibility, it suffices to show that every $\Ad(\Hh)$-module is also an $\Ad(\K)$-module. To that end, let $\vg \subset \pg$ be an $\Ad(\Hh)$-module, and let $\ag_\vg := \vg \cap \ag$. We claim that 
\[
    \vg = \bigcup_{k\in \K} \Ad(k) \ag_\vg.
\]
To see that, first observe that $\ag_\vg$ is a maximal abelian subspace in $\vg$. Indeed, if $\bg \subset \vg$ is abelian then by uniqueness of maximal abelian subspaces in $\pg$ \cite[Thm.~6.51]{Knapp2e}, there exists $k\in \K$ such that $\Ad(k)\bg \subset \ag$. Writing $k = m h$ with $h\in \Hh$, $m\in \M$, and using that $\Ad(\M)$ acts trivially on $\ag$, we get $\Ad(h) \bg \subset \ag$. But $\vg$ is $\Ad(\Hh)$-invariant, hence $\Ad(h)\bg \subset \ag \cap \vg = \ag_\vg$. This argument shows furthermore that 
\[
    \vg = \bigcup_{h\in \Hh} \Ad(h) \ag_\vg.
\]
Now any $k\in \K$ may be written as $k = h m$ with $\Ad(m)|_\ag = \Id_\ag$, hence $\Ad(k) \ag_\vg = \Ad(h)\ag_\vg$, and the claim follows.

Finally, if we had $\pg_i \simeq \pg_j$ as $\Ad(\Hh)$-modules for some $i\neq j$, then there would be more irreducible $\Ad(\Hh)$-modules other than the $\pg_i$'s. These would also be irreducible as $\Ad(K)$-modules by argument in the previous paragraph, and this is a contradiction because the $\pg_i$'s are inequivalent as $\Ad(K)$-modules.
\end{proof}

\section{The orbits are symmetric spaces}\label{sec_Lorbits}

The upshot of Proposition \ref{prop_Qtransitive} is that the stabiliser $\Hh = \Ll_p$ of a point $p$ satisfying \eqref{eqn_aperpn} must be large enough so that $\Q\Hh = \Ll$ holds. This condition significantly constrains the geometry of the $\Ll$-orbits, as they are diffeomorphic to $\Ll/\Hh$. For example, if $\Ll$ is split semisimple (i.e.~$\mg = 0$, satisfied for instance when $\Ll = \Sl_n(\RR)$), then the condition $\qg + \hg = \lgo$ forces $\Hh = \K$ to be a maximal compact subgroup, and the $\Ll$-orbits must be symmetric spaces. Our aim in this section is to show that this holds even without the split semisimple assumption: 

\begin{proposition}\label{prop_Ldotpsym}
The $\Ll$-orbits are symmetric spaces.
\end{proposition}

To prove this, it is important to first deduce a geometric application of \Cref{cor_AdHAdK}. Recall the Riemannian submersion $\pi^\N : M\to \N\backslash M$. Since  $\N$ is a normal subgroup of $\Q$, we have the corresponding Lie group  and Lie algebra morphisms
\[
  \varphi : \Q \to \Q/\N, \qquad \Phi := d \varphi|_e : \qg \to \qg/\ngo.
\] 
There is also a natural action of $\Q$ on $\N\backslash M$: 
\[
  q\cdot (\N \cdot x) := \N \cdot (q\cdot x), \qquad q\in \Q,  \quad  x \in M,
\]
for which by definition $\pi^\N$ is $\Q$-equivariant. 
The ineffective kernel of said action contains $\N$, thus we have an isometric, cocompact action of $\Q/\N$ on  $\N\backslash M$, given by
\[
    \varphi(q) \cdot (\N\cdot x) := \N \cdot (q\cdot x), \qquad \hbox{or equivalently,} \qquad \varphi(q) \cdot \pi^\N(x) = \pi^\N(q\cdot x).
\] 
In order to simplify the notation, given $Z\in \qg$  we denote $\check Z := \Phi(Z)$.

\begin{lemma}
Let $Z\in  \qg$  and let $Z^* \in \Gamma(TM)$ be the corresponding Killing field. Then,
\[
    (\pi^\N)_* Z^* = \check Z^*,
\] 
where $(\check Z)^*$ is the  Killing field on $\N\backslash M$ associated to $\check Z = \Phi(Z)\in \qg / \ngo$.
\end{lemma}

\begin{proof}
Let $p\in M$, $b:= \pi^\N (p) \in \N\backslash M$. If $\exp_\Q : \qg \to \Q$, $\exp_{\Q/\N} : \qg /\ngo \to \Q/\N$ denote the corresponding Lie exponential maps, then we have $\exp_{\Q/\N} \circ \,  \Phi = \varphi \circ \exp_\Q$. Using this, we obtain:
\begin{align*}
  (\check Z)^*_b &= \ddt\big|_0 \exp_{\Q/\N} (t \check Z) \cdot b = \ddt\big|_0  \varphi(\exp(t Z)) \cdot b \\
    & = \ddt\big|_0  \pi^\N( \exp(tZ) \cdot p ) =  (\pi^\N)_* Z_p^* .
\end{align*}
\end{proof}
 Notice that the above in particular implies that $\ho^\N Z^*$ is a basic vector field.

\begin{lemma}\label{lem_RicZZ}
Let $Z\in \mg$ and let $Z^* \in \Gamma(TM)$ be the corresponding Killing field. Then, the shape operator of the $\N$-orbits in $M$ in the direction of $Z^*$ vanishes, that is, $L_{Z^*} = 0$. In particular, 
\begin{equation}\label{eqn_RicZcheck}
    - \Vert \check Z^* \Vert^2 = \Ric_{\N\backslash M}(\check Z^*, \check Z^*).
\end{equation}
\end{lemma}

\begin{proof}
Let $p\in M$ satisfy \eqref{eqn_aperpn}. By  \Cref{cor_kperpp}, for each $Z\in \mg$, $Z^*$ is $\N$-horizontal at $p$. We then have
\[
    \la L_{Z^*} U^*, U^*\ra_p = -\la  \nabla_{U^*} U^*, Z^* \ra_p = \la [Z^*, U^*], U^* \ra_p,
\]
for all $U\in \ngo$, where the last equality follows from \cite[(7.27)]{Bss}. We now claim that $(\ad Z)|_\ngo : \ngo \to \ngo$ is skew-symmetric  with respect to the inner product induced by the linear isomorphism $\ngo \simeq \ngo \cdot p$. To see that, \Cref{cor_AdHAdK} implies that the inner product induced by $g$ on $\pg \cdot p \subset  T_p (\Ll\cdot p)$ is $\Ad(\K)$-invariant. Since $\ngo \cdot p \subset \pg \cdot p$ by  \Cref{cor_kperpp}, and $\ngo$ is $\Ad(\M)$-stable, it follows that $(\ad Z)|_\ngo$ is skew-symmetric for each $Z\in \mg$. In other words, $L_{Z^*} = 0$. Plugging this into \Cref{thm_Naction}, \ref{item_RicP} yields \Cref{eqn_RicZcheck}.
\end{proof}

\begin{proof}[Proof of Proposition \ref{prop_Ldotpsym}] 
Since $\Ll =  \K \A \N$, $\Q = \M_0 \A \N$, $\M \leq \K$, and $\Ll = \Q \Hh$ (\Cref{prop_Qtransitive}), to prove that $\Hh = \K$ is maximal compact we must show that $\M \leq \Hh$, or equivalently, $\Q / \N = \G /\N$. To that end, we let $Z\in \mg$, 
and consider the smooth function 
\[
    f := \unm \left\Vert \check Z^* \right\Vert^2 : \N\backslash M \to \RR,
\]
where $\check Z = \Phi(Z) \in \qg /\ngo$.
For all $X\in \Gamma(T(\N\backslash M))$ we have the well-known Bochner formulas 
\[
    \la \nabla f, X \ra = - \left\la \nabla_{\check Z^*} \check Z^*, X  \right\ra, \qquad \Delta_{\N\backslash M} f =  \big\Vert \nabla^{\N\backslash M} \check Z^* \big\Vert^2 - \Ric_{\N\backslash M}(\check Z^*, \check Z^*).
\]
The first formula and the fact that $[\mg,\ag] = 0$ imply that $f$ is constant along $\G/\N$-orbits. Since $(\G/\N) \backslash (\N\backslash M) \simeq \G \backslash M$ is compact, $f$ must admit a global maximum  at  $x\in \N\backslash M$, hence
\[
   \left( \Delta_{\N\backslash M} f \right) |_x \leq 0.
\] 
On the other hand, the second formula and Lemma \ref{lem_RicZZ} yield 
\[
    \Delta_{\N\backslash M} f = \Vert \nabla^{\N\backslash M} \check Z^* \Vert^2 + \Vert \check Z^* \Vert^2 \geq 0.
\]
In particular, we have $(\nabla \check Z^*)|_x = 0$ and $\check Z^*_x = 0$. Since these determine a Killing field completely, we conclude that $\check Z = 0$. This also gives $Z = 0$, as $\Phi|_{\mg\oplus \ag}$ is injective.
\end{proof}

Next, recall the definition of the \emph{polar group}:  $\Pi := \Ll_{\Sigma} / \Hh$. Here $\Ll_\Sigma$ is the stabiliser of $\sigma(\Sigma_\Ll)$ and $\Hh := \Ll_p$ is the isotropy at $p$, which is also precisely the set of elements in $\Ll$ fixing $\sigma(\Sigma_\Ll)$ pointwise. Recall also that, by \cite[Prop.~1.3]{GZ12}, $\Pi$ is a discrete subgroup of $N(\Hh) / \Hh$ which  acts on $\Sigma_\Ll$  isometrically and properly discontinuously, making $\sigma$ equivariant, and so that the metric spaces $\Sigma_\Ll/\Pi$ and $\Ll\backslash M$ are isometric.

\begin{corollary}\label{cor_sectioncpt}
The polar group $\Pi$ is trivial and the section $\Sigma_\Ll$ is compact.
\end{corollary}

\begin{proof}
By \Cref{prop_Ldotpsym} $\Hh = \K$ is a maximal compact subgroup of the semisimple Lie group $\Ll$. Since  $\K$ is self-normalising (\Cref{prop_N(K)=K}), 
 $\Pi \leq N(\K)/\K$ must be trivial. In particular, $\Sigma_\Ll \simeq \Ll\backslash M$ is compact. 
\end{proof}


\section{Isometric splitting}

We are finally in a position to prove Theorem \ref{thm_unimod}. We continue using the same notation as in previous sections, with one slight change: since by \Cref{cor_sectioncpt} the submanifold $\Sigma_\Ll$  is embedded, we will identify it with its image in $M$ and simply assume that $\sigma : \Sigma_\Ll \hookrightarrow M$ is the inclusion map. As in previous sections, we may and will also assume that the points in $\Sigma_\Ll$ satisfy the technical condition \eqref{eqn_aperpn}. 

 Given $U\in \lgo$, consider the smooth function
\[
    f_U := \unm \Vert U^* \Vert^2 : \Sigma_\Ll \to \RR.
\]

\begin{lemma}\label{lem_Laplacef_U}
Along $\Sigma_\Ll$ we have that 
\[
 \Delta_{\Sigma_\Ll} f_U - \la \nabla f_U, \mcv^\Ll \ra = \ric^\Ll(U^*,U^*) + \Vert U^* \Vert^2 + 2 \, \Vert L U^*\Vert^2.
\]
\end{lemma}

\begin{proof}
For $X\in \Gamma(T \Sigma_\Ll)$ we have
\[
    \la \nabla f_U, X\ra = \la \nabla_X U^*, U^* \ra = -\la X, \nabla_{U^*} U^* \ra.
\]
We now use an local, $\Ll$-basic (i.e.~horizontal and projectable), orthonormal frame $\{X_k\}$ for $\ho^\Ll$ satisfying $\nabla_{X_k} X_k=0$ at the point:
\begin{align*}
 \Delta_{\Sigma_\Ll} f_U =& -\sum_k \la \nabla_{X_k} \ho \nabla_{U^*} U^*, X_k \ra = -\sum_k X_k \la \nabla_{U^*} U^*, X_k \ra \\
    =& \sum_k X_k \la  L_{X_k} U^*, U^*\ra \\
    =& \sum_k \left( \left\la \big(\nabla_{X_k} L \big)_{X_k} U^*, U^* \right\ra + \left\la L_{X_k} (\nabla_{X_k} U^*), U^* \right\ra + \la  L_{X_k} U^*, \nabla_{X_k} U^* \ra  \right) \\
    =& \sum_k  \left\la \big(\nabla_{X_k} L \big)_{X_k} U^*, U^* \right\ra  + 2 \, \Vert L U^* \Vert^2,
\end{align*}
where in the last step we have used that basic vector fields commute with vertical Killing fields \cite[Lemma 3.5]{alek_sol}. The lemma now follows from the vertical Einstein equation \eqref{eqn_EV} with $\Ll$-symmetry, and the fact that $A^\Ll = 0$ (\Cref{prop_Lpolar}).
\end{proof}

\begin{proof}[Proof of Theorem \ref{thm_unimod}]
By \Cref{thm_ricneg}, $\Ll$ is semisimple and, by \Cref{lem_centerless}, it may be assumed to have finite center.

By Proposition \ref{prop_Ldotpsym}, for each simple factor $\Ll_i \leq \Ll$ there exists $\lambda_i < 0$ such that for any $U_i \in \lgo_i$ we have
\begin{equation}\label{eqn_Ricuiui}
        \ric^\Ll (U_i^*, U_i^*) = \lambda_i \, \Vert U_i^*\Vert^2.
\end{equation}
Since $\ngo = \ngo_1 \oplus \cdots \oplus \ngo_r$, $\ngo_i = \lgo_i \cap \ngo$, \Cref{cor_tr_NRicL} gives 
\begin{equation}\label{eqn_sum_lambda_i}
   \sum_i (\dim \ngo_i) \lambda_i =  \tr_{\ve^\N} \Ric^\Ll = - \dim \ngo =  - \sum_i (\dim \ngo_i).    
\end{equation}

On the other hand, using \eqref{eqn_Ricuiui}, Lemma \ref{lem_Laplacef_U} reads as 
\begin{equation}\label{eqn_Deltaf_ui}
     \Delta_{\Sigma_\Ll} f_{U_i} - \la \nabla f_{U_i}, \mcv^\Ll \ra = (1+\lambda_i) \Vert U_i^* \Vert^2 + 2 \, \Vert L U_i^*\Vert^2.
\end{equation}
Since $\Sigma_\Ll$ is compact (\Cref{cor_sectioncpt}), by the maximum principle applied to \eqref{eqn_Deltaf_ui} we must have that $\lambda_i \leq -1$ for all $i$.  \Cref{eqn_sum_lambda_i} then yields $\lambda_i = - 1$ for all $i = 1, \ldots, r$.

Putting this back into \eqref{eqn_Deltaf_ui}, from the maximum principle we deduce that $LU^*=0$ for all $U\in \lgo$. That is, the $\Ll$-orbits are totally geodesic. Since $A^\Ll = 0$, at least locally $M$ splits isometrically as the product of the compact Einstein manifold $\Sigma_\Ll$ and the Einstein symmetric space $\Ll/\K$, both with Einstein constant $-1$. Moreover, since the polar group is trivial by \Cref{cor_sectioncpt}, the section intersects each orbit exactly once, and hence the splitting is global. 
\end{proof}

\section{Ricci flat manifolds with cocompact symmetry}\label{sec_Ricciflat}

The following argument is essentially contained in the proof of \cite[Thm.~3]{ChGr71}. Recall that a \emph{line} is a geodesic $\gamma: \RR \to M$ which realises the distance between any pair of its points.

\begin{lemma}\label{lem_line}
Let $(M^n,g)$ be a complete, non-compact Riemannian manifold for which $\Isom(M^n,g)$ acts cocompactly. Then, $(M^n,g)$ contains a line.
\end{lemma}

\begin{proof}
Let $p\in M$. Given a sequence $\{q_n\}_{n\in \NN}$ such that $d(p,q_n) = n$ for all $n\in \NN$, choose minimising geodesics $\gamma_n$ with $\gamma_n(0) = p$ and $\gamma_n(n)=q_n$, for each $n\in \NN$. After passing to a subsequence we may assume that $(\gamma'_n(0))_{n\in\NN}$ converges to a unit vector $v\in T_p M$. It is then clear that the geodesic $\gamma$ with $\gamma(0) = p$, $\gamma'(0) = v$, is a \emph{ray}: it minimises the distance between $\gamma(t)$ and $\gamma(s)$ for all $t,s \geq 0$. 

Let  $K \subset M$ be a compact fundamental domain for $\Isom(M,g)$. For each $n\in \NN$, there exists an isometry $f_n \in \Isom(M,g)$ such that $f_n(\gamma(n)) \in K$. By compactness, after passing to a subsequence we may assume that $f_n(\gamma(n)) \to q \in K$ and  $(f_n)_* \gamma'(n) \to w \in T_q M$. Notice that for each $n\in \NN$, the geodesic $f_n\circ \gamma$ is minimising on $(-n,\infty)$. It follows that the geodesic $\sigma : \RR \to M$ with $\sigma(0) = q$, $\sigma'(0) = w$ is a line.
\end{proof}

\begin{theorem}\label{thm_Ricciflat}
Let $(M^n,g)$ be a complete, Ricci flat Riemannian manifold, and let $\G$ be a Lie group acting properly, isometrically and cocompactly on it. Then, the universal cover  splits isometrically as a Riemannian product $\tilde M = \bar M \times \RR^k$, with $\bar M$ compact.  In particular, $\G_0 \subset \Isom(\RR^k)$ if the action is effective.
\end{theorem}

\begin{proof}
The universal cover $\tilde M$ is Ricci flat with respect to the pull-back metric $\tilde g$, and the action of $\G$ lifts to a proper isometric action on $\tilde M$. Also, the fundamental group $\pi_1(M)$ also acts --via deck transformations-- properly and isometrically on $\tilde M$. Since $M/\G$ is compact, it follows that $\Isom(\tilde M, \tilde g)$ acts cocompactly on $\tilde M$. If $\tilde M$ is compact, then we are done. Otherwise, by \Cref{lem_line}, $\tilde M$ contains a line, thus by the Cheeger-Gromoll splitting theorem \cite{ChGr71} we have an isometric splitting $\tilde M \simeq \tilde M_1 \times \RR^1$. In particular, $\tilde M_1$ is Ricci flat. If $\tilde M_1$ contains a line, we split again. In this way, we may continue applying the splitting theorem until we reach a point where $\tilde M \simeq \tilde M_k \times \RR^k$, with $\tilde M_k$ containing no lines. 

We now claim that $\Isom(\tilde M_k)$ acts cocompactly on $\tilde M_k$. To see this, notice that all the lines in $\tilde M$ are contained in the $\RR^k$ factor. Since isometries map lines to lines, it follows that any isometry of $\tilde M$ must map the $\RR^k$ factor to itself, and hence also preserve $\tilde M_k$. This implies that $\Isom(\tilde M) \simeq \Isom(\tilde M_k) \times \Isom(\RR^k)$. The claim thus follows from the fact that $\Isom(\tilde M)$ acts cocompactly on $\tilde M$. By \Cref{lem_line} we conclude that $\tilde M_k$ is compact.

Regarding the claim about $\G_0$, it will follow once we prove that $\Isom(\tilde M_k)$ is discrete. But if $X$ is a non-zero Killing field then by Bochner's theorem, $X$ must be parallel ($\tilde M_k$ is compact and with non-positive Ricci curvature). Since $\tilde M_k$ is also simply-connected, this would imply that it splits an $\RR$ factor, thus contradicting compactness. 
\end{proof}

\begin{remark}
In the case of a cocompact isometric action by a \emph{connected} Lie group, completeness follows. However, in \Cref{thm_Ricciflat} we do not assume that $\G$ is connected, and we are thus  forced to assume completeness.
\end{remark}

\begin{appendix}

\section{Proper isometric actions and polar actions}\label{app_isom_actions}



Let $\G$ be a Lie group acting isometrically on a complete Riemannian manifold $(M^n,g)$ via the Lie group homomorphism
\[
    \tau : \G \to \Isom(M^n,g).
\]
After changing $\G$ by $\G/\ker \tau$, we assume from now on that the action is effective, i.e.~$\tau$ is injective. If $\G$ is non-compact, it is natural to assume that the action is \emph{proper}, that is, the `shear map'
\[
        \G \times M \to M \times M, \qquad (g,p) \mapsto  (\tau(g)p, p),
\]
is a proper continuous function. (Since $M$ is locally compact and Hausdorff, the above definition is equivalent to the notions of \emph{Borel properness}, see \cite[Thm.~1.2.9 (5)]{Pal61}, and to \emph{Palais properness}, see \cite[Def.~1.2.2]{Pal61}.) For isometric actions there is a simple criterion:

\begin{lemma}\label{lem_proper}
The action of $\G$ is proper if and only if $\tau(\G)$ is a closed subgroup of $\Isom(M^n,g)$.
\end{lemma}

\begin{proof}
Let $\{(\tau(g_k) p_k, p_k)\}_k$ be a sequence contained in a compact set in $M\times M$. Then up to passing to subsequences, as $k\to \infty$ we have 
\begin{equation}\label{eqn_pqinfty}
        p_k \to p_\infty, \qquad \tau(g_k) p_k \to q_\infty. 
\end{equation}
The triangle inequality yields
\[
    d(p_k,p_\infty) = d(\tau(g_k)p_k, \tau(g_k) p_\infty) \geq 
    | d(q_\infty, \tau(g_k)p_\infty) - d(\tau(g_k) p_k, q_\infty)  | \geq 0.
\]
By \eqref{eqn_pqinfty} it follows that $\tau(g_k) p_\infty \to q_\infty$ as $k\to \infty$. Thus, the isometries $\tau(g_k)$ subconverge to some $f\in \Isom(M^n,g)$, since $(M^n,g)$ is complete. By assumption we have $f = \tau(g_\infty)$ for some $g_\infty \in \G$, and it follows that $(g_k,p_k) \to (g_\infty, p_\infty)$ as desired. 

Conversely, if the action is proper and $\tau(g_k) \to f$ for some sequence $\{g_k\}_k$ in $\G$, then 
\[
    (\tau(g_k)p,p) \to (f(p),p),
\]
hence $\{ g_k \}_k$ subconverges to some $g_\infty \in \G$, and $f = \tau(g_\infty)$.
\end{proof}

In particular, for all $p\in M$ the isotropy group $\G_p$ is compact. In addition,  one of the main advantages of a proper Lie group action is that the Slice Theorem holds \cite{Pal61}.  Recall that an orbit $\G \cdot p$ is called:
\begin{itemize}
    \item  \emph{principal}, if the isotropy group $\G_p$ is not properly contained  in any other isotropy group $\G_q$, $q\in M$; $\G_p$ is then called a principal isotropy group;
    \item \emph{exceptional}, if $\G_p \lneq \G_q$ for some principal isotropy group $\G_q$, with $\G_q / \G_p$ discrete;
    \item \emph{singular}, if $\dim \G_p < \dim \G_q$ for some principal isotropy group $\G_q$.
\end{itemize}
In general, the orbit space $\G\backslash M$ is a possibly singular metric space. An important  consequence of the Slice Theorem is that there is, up to conjugation, a unique principal isotropy group, and the union of all principal orbits forms an open dense subset of full measure whose quotient is a convex subset of $\G\backslash M$ and a smooth manifold. We set 
\[
    M_{\sf sing} := \{ p\in M : \G\cdot p \hbox{ is a singular orbit} \}, \qquad M_{\sf reg} := M \setminus M_{\sf sing}.
\]
Notice that $M_{\sf reg}$ contains not only principal orbits but also the exceptional ones. It will become clear after $\S$\ref{sec_quo_fb} that $M_{\sf reg}$ is also open in $M$.

In certain situations, there exists a globally-defined slice:   a proper isometric action is called \emph{polar}, if there exists an immersed submanifold $\sigma : \Sigma \to M$ --called a \emph{section}-- intersecting \emph{all} orbits orthogonally. By a result of Boualem  \cite{Boua95} (see also \cite{HLO,GZ12}), this is equivalent to the vanishing of the integrability tensor $A$ of the Riemannian submersion defined on the regular part:
\[
  \pi^\G : M_{\sf reg} \to \G\backslash M_{\sf reg} =: B, \qquad p \mapsto \G \cdot p.
\]
Next, recall the definition of the \emph{polar group}:  $\Pi := \G_{\Sigma} / \Hh$. Here, $\G_\Sigma = \{ f\in \G : f \cdot \sigma(\Sigma) = \sigma(\Sigma) \}$ is the stabiliser of $\sigma(\Sigma)$, and $\Hh := \G_p$ is the isotropy at $p$, which is also precisely the set of elements of $\G$ fixing $\sigma(\Sigma)$ pointwise. Recall also that, by \cite[Prop.~1.3]{GZ12}, $\Pi$ is a discrete subgroup of $N(\Hh) / \Hh$ which  acts on $\Sigma$  isometrically and properly discontinuously, making $\sigma$ equivariant, and so that the metric spaces $\Pi\backslash\Sigma$ and $\G\backslash M$ are isometric.

\section{The frame bundle}\label{sec_FM}

The material contained in this section is standard, but we have decided to include it for the reader's convenience, and to introduce our notation.

\subsection{Basic definitions}\label{sec_FM_basic} Given a smooth Riemannian manifold $(M^n,g^M)$, its  \emph{frame bundle} is the set 
\[
    FM = \{ u_p \mid p\in M, u_p : \RR^n \to T_p M \hbox{ is a linear isometry}\}.
\] 
We refer to each $u_p$ as a frame at $p$. $FM$ has a natural structure of an $\Or(n)$-principal bundle over $M$ with the obvious projection map $\pi: FM \to M$, $\pi(u_p) = p$, and where the (right) action of $h\in \Or(n)$ on $FM$ is simply given by composition $u_p \, h := u_p \circ h$. The latter makes sense since by definition $h : \RR^n \to \RR^n$ is a linear isometry. 

The Levi-Civita connection of $g^M$ is an affine connection on $TM$, thus it gives rise to a principal connection on $FM$, that is, a decomposition
\begin{equation}\label{eqn_ppalconnection}
    T(FM) = \hat\ve^{\Or} \oplus \hat\ho^{\Or} , \qquad  \hat \ve^{\Or(n)} = \ker d\pi.
\end{equation}
More explicitly, a curve $u(t)$ in $FM$ is horizontal (i.e.~ its velocity always lies in $\hat \ho^\Or$) if and only if the corresponding frames move by parallel transport along the curve $\pi(u(t))$ in $M$. 

After fixing a bi-invariant metric $g^{\Or}$ on $\Or(n)$, defined by an $\Ad(\Or(n))$-invariant inner product on $\sog(n)$, we obtain an $\Or(n)$-invariant metric $\hat g^{\Or}$ on the vertical distribution by means of identifying $\hat \ve^{\Or} \simeq \sog(n)$ via fundamental vector fields. We also have the important pull-back metric $\pi^* g^M$ on $FM$, which is only positive semi-definite as it vanishes on  $\hat \ve^\Or$, but  still highly relevant to the geometry on $M$. From these, we construct the (positive-definite) lifted Riemannian metric on $FM$:
\[
    \hat g := \hat g^{\Or} + \pi^* g^M ,
\]

Notice that, by definition, $\hat g$ is constant on fundamental vector fields. That is, if $U,V \in \sog(n)$ and $\hat U^*,  \hat V^*$ denote the corresponding fundamental vector fields on $FM$, then 
$\hat g(\hat U^*, \hat V^*)_u = g^\Or(U,V)$ is independent of $u\in FM$. In particular, 
\[
    \hat g(\nabla_{\hat U^*} \hat V^*, X) = - \unm X \hat g(\hat U^*, \hat V^*) = 0
\]
for all $X\in \hat \ho^\Or$, from which it follows that the $\Or(n)$-orbits in $FM$ are totally geodesic.

\subsection{Lifting isometries}\label{sec_liftiso} Let $f: M \to M$ be an isometry. Then for each $p\in M$, $df_p : T_p M \to T_{f(p)} M$ is a linear isometry. We define $\hat f : FM \to FM$ via
\[
    \hat f (u_p) := df_p \circ u_p.
\]
This is clearly a smooth map, and it lifts $f$ in the sense that
\[
    \pi \circ \hat f = f \circ \pi.        
\] 
Since isometries preserve parallel transport, it is also clear that $\hat f_*$ preserves $\hat\ho^\Or$. Moreover, using that $\hat f$ commutes with the principal $\Or(n)$-action, it follows that $\hat f_*$ also preserves $\hat \ve^\Or$ and that $\hat f_* |_{\hat \ve^\Or} = \Id_{\hat \ve^\Or}$. Hence, 
\[
        \hat f^* (\hat g) = \hat g^\Or + \hat f^* \pi^* g^M = \hat g^\Or + \pi^* f^* g^M = \hat g^\Or + \pi^* g^M  = \hat g,
\]
that is, $\hat f$ is an isometry of $\hat g$.

The key observation at this point is that the lifted isometry $\hat f$ has no fixed points, unless $f = \Id_M$ (in which case also $\hat f = \Id_{FM}$). Indeed, if $\hat f (u_p) = u_p$ then $f(p) = p$ and $df_p \circ u_p = u_p$ which implies $df_p = \Id_{T_p M}$ since $u_p$ is invertible. But an isometry is completely determined by its first jet, thus $f = \Id_M$.

\subsection{Quotients and the frame bundle}\label{sec_quo_fb}

The main reason  to consider the frame bundle is that, given a proper isometric action of a Lie group $\N$ on $(M^n,g)$, the corresponding lifted action on $(FM, \hat g)$ is free and isometric, and hence $\N \backslash FM$ is a smooth Riemannian manifold. Since the lifted action of $\N$ and the principal $\Or(n)$-action commute, the latter sends $\N$-orbits to $\N$-orbits in $FM$, thus it induces an action of $\Or(n)$ on $\N\backslash FM$ which happens to be by isometries, with orbit space  $\N \backslash M$:
\begin{equation}\label{eqn_cdFM}
\begin{tikzcd}
  FM \arrow[r,"\hat\pi^\N"] \arrow[d,"\pi"]
    & \N\backslash FM \arrow[d, "\check\pi"] \\
  M \arrow[r,"\pi^\N"]  &  \N \backslash M 
\end{tikzcd}
\end{equation}

Notice that for each $u_p\in FM$, $\N\cdot u_p \cap \Or(n)\cdot u_p = \N_p \cdot u_p$, where $\N_p$ denotes the isotropy group at $p$ (for the action on $M$). Thus, the $\N$- and $\Or(n)$-orbits through $u_p$ are transversal if and only if $\N_p$ is discrete.

For each $U\in \ngo$ we denoted by $U^* \in \Xg(M)$ the corresponding Killing field, and let us denote by $\hat U^* \in \Xg(FM)$ the Killing field for the lifted action. Observe that $\hat U^*$ and $U^*$ are $\pi$-related, and hence
\[
    (\pi^* g^M)\left(\hat U^* ,\hat U^*\right) =  g^M(U^*, U^*). 
\]

Let  us assume that $\N$ has discrete principal isotropy groups. Let $M_{\sf sing}$ denote the union of singular orbits, a closed, stratified subset of $M$.
If $\iota : M_{\sf sing} \to M$ is the inclusion, the corresponding pull-back bundle $\iota^* FM$  is the $\Or(n)$-bundle $FM_{\sf sing} \to M_{\sf sing}$ given by 
\[
    FM_{\sf sing} = \{ u_p \in FM : \hat \ve^\N (u_p) \cap \hat \ve^\Or(u_p) \neq 0  \}.
\] 
Since $\hat\ve^\N$ and $\hat\ve^\Or$ are vector subbundles of $T(FM)$, it is clear that $FM_{\sf sing}$ is closed in $FM$. Moreover, because the actions of $\N$ and $\Or(n)$ commute, both vector subbundles are $\Or(n)$-invariant, and hence so is $FM_{\sf sing}$. Notice that $u\in FM_{\sf sing}$ if and only if the orbit $\N\cdot \pi(u)$ has dimension strictly smaller than the principal orbits. We also set
\[
   FM_{\sf reg} :=  FM \backslash FM_{\sf sing} = \left\{ u_p :  (\pi^* g^M)|_{ \hat \ve^\N (u_p)} > 0 \right\},
\]
which is  a principal $\Or(n)$-bundle over $M_{\sf reg}$.

\section{Some results on semisimple Lie groups and algebras}\label{app_sslie}

In this appendix we prove three important properties of semisimple Lie groups and algebras for which we were not able to find a direct reference in the Lie theory literature.

\begin{proposition}\cite{Smi35}\label{prop_fin_gen}
The center of a connected semisimple Lie group is a finitely-generated abelian group.
\end{proposition}

\begin{proof}
It suffices to prove this under the assumption that the Lie group $\Ll$ is simply-connected, for the center of a Lie group is a quotient of the center of its universal cover, and quotients of finitely-generated groups are finitely generated. The center $\Zz$ of a simply-connected Lie group $\Ll$ is isomorphic to the fundamental group of the quotient group $\Ll/\Zz$, and it is shown in \cite{Smi35} that the latter is finitely generated.
\end{proof}

Consider now a Cartan involution $\theta$ for the semisimple Lie algebra $\lgo$, with $\lgo = \kg \oplus \pg$  the corresponding Cartan decomposition. There is a corresponding decomposition at the group level, $\Ll = \K \exp(\pg)$, where $\K \leq \Ll$ is the connected Lie subgroup with Lie algebra $\kg$. 

\begin{proposition}\label{prop_N(K)=K}
If $\Ll$ is a connected semisimple Lie group then $N_\Ll(\K) = \K$.
\end{proposition}
\begin{proof}
Let $\exp(A) \in N_\Ll(\K) \setminus \K$, $A\in \pg$. For each $k\in \K$ we have 
\[
  \exp(A) k \exp(-A) = \tilde k \in \K,
\]
from which
\[
   \exp(\Ad(k^{-1})A) =   k^{-1} \exp(A) \, k  = k^{-1} \, \tilde k \exp(A)
\]
and hence $k = \tilde k$ and $\Ad(k) A = A$ by the fact that the global Cartan decomposition gives a diffeomorphism $\Ll \simeq \K \times \pg$. This implies that $[\kg,A] = 0$. Using the inner product $B_\theta$ we deduce that for all $C\in \pg$ and $Z\in \kg$ it holds that
\[
    0 = B_\theta([Z, A], C) = - B_\theta(Z, [A,C]),
\]
hence $[A,C]\in \kg$ vanishes. This yields $A\in \zg(\lgo)$, from which $A=0$ and the proposition follows.
\end{proof}

 Given any maximal abelian subspace $\ag \subset \pg$ and a choice of positive roots $\Lambda^+ \subset \Lambda \subset \ag^*$, we consider the corresponding Iwasawa decomposition 
\[
  \lgo = \kg \oplus \ag \oplus \ngo, \qquad \ngo = \bigoplus_{\lambda \in \Lambda^+} \lgo_\lambda,
\] 
where $\lgo_\lambda$ denotes the root space corresponding to the root $\lambda \in \Lambda$. We  denote by $\mg$ the centraliser of $\ag$ in $\kg$. Recall that $\lgo_0 = \mg\oplus \ag$ normalises $\ngo$. Finally, we consider on $\lgo$ the $\Ad(\Ll)$-invariant inner product $B_\theta(\cdot, \cdot) := -B(\cdot, \theta \cdot )$, where $B$ denotes the Killing form of $\lgo$.

\begin{proposition}\label{prop_nilspan}
Let $\lgo$ be a semisimple Lie algebra without compact simple factors. Then, $\lgo$ is the linear span of all  Borel nilradicals, when varying over all possible Iwasawa decompositions.
\end{proposition}


\begin{proof}
Recall that $\lgo =\theta \ngo \oplus \mg \oplus \ag \oplus \ngo$, where $\theta\ngo$ is the Borel nilradical corresponding to the opposite choice of positivity for the roots. It remains to be shown that the linear span of Borel nilradicals contains $\mg\oplus\ag$.

Given $x\in \Ll$, it is well-known that the decomposition
\[
    \lgo = \Ad(x)\kg \oplus \Ad(x)\ag \oplus \Ad(x) \ngo
\]
is another Iwasawa decomposition. Thus, it suffices to show that the linear span of $ \{ \Ad(x)\ngo : x\in \Ll\}$ contains $\mg\oplus\ag$. To that end, we will prove that 
\[
  \mg\oplus\ag \subset [\ngo,\theta\ngo],
\]
which is clearly enough. 

Suppose  that  $Z\in \mg\oplus \ag = \lgo_0$ is such that $Z\notin [\ngo,\theta\ngo]$.
By the Jacobi identity and the fact that $\theta\in \Aut(\lgo)$, the subspace $[\ngo,\theta\ngo]$ is $\ad(\ag)$-stable and $\theta$-stable, hence it decomposes as a direct sum
\[
    [\ngo,\theta \ngo] = \left([\ngo,\theta\ngo]\cap \mg \right) \oplus \left([\ngo,\theta\ngo] \cap \ag \right) \oplus \bigoplus_{\lambda \in \Lambda} [\ngo,\theta\ngo] \cap \lgo_\lambda.
\] 
We may therefore assume without loss of generality that either $Z\in \mg$ or $Z\in \ag$. Since the Killing form $B$ is definite on each of $\mg$ and $\ag$, and it makes the root spaces orthogonal, we may also assume  that 
$Z \perp_B [\ngo,\theta \ngo]$. Thus, for each $X, Y\in \ngo$ we have 
\[
   0 = B(Z, [X, \theta Y]) = B([Z,X], \theta Y) = - B_\theta([Z,X], Y).
\]
This implies that $[Z,\ngo] = 0$ and hence $Z=0$ by Lemma \ref{lem_Z_mg-ngo} below, a contradiction. 
\end{proof}

\begin{lemma}\label{lem_Z_mg-ngo}
Under the asumptions of \Cref{prop_nilspan}, we have that $Z_{\mg}(\ngo) = Z_{\ag}(\ngo) =  0$. 
\end{lemma}

\begin{proof}
Let $Z\in \mg$ with $[Z,\ngo] = 0$. Then $[Z, \theta \ngo] = [Z,\ag] =  0$ as well. Since $\pg \subset \ag \oplus \ngo \oplus \theta\ngo$, we have $[Z,\pg] = 0$. By the Jacobi identity this implies that $[Z, \pg + [\pg,\pg] ] = 0$. Now using the bracket relations from the Cartan decomposition $\lgo = \kg\oplus \pg$, it is easy to see that $\pg + [\pg,\pg]$ is an ideal in $\lgo$, which intersects all non-compact simple factors. Since $\lgo$ has no compact simple factors, $\lgo = \pg+[\pg,\pg]$, and hence $Z\in \zg(\lgo) = 0$. The proof for $Z_{\ag}(\ngo)$ is completely analogous.
\end{proof}

\end{appendix}

              \bibliography{C:/Users/cboehm/Documents/Documents/Tex/ramlaf2}
\bibliographystyle{amsalpha}

\end{document}